\documentclass{amsart}
\usepackage{geometry}
\usepackage{amssymb,amsmath,latexsym,amsthm, verbatim}

\usepackage[colorlinks=true,urlcolor=blue, citecolor=red,linkcolor=blue,
linktocpage,pdfpagelabels, bookmarksnumbered,bookmarksopen]{hyperref}
\usepackage[hyperpageref]{backref}

\numberwithin{equation}{section}

\newtheorem{teo}{Theorem}[section]
\newtheorem*{main}{Main Theorem}
\newtheorem{lm}[teo]{Lemma}
\newtheorem{prop}[teo]{Proposition}
\newtheorem{coro}[teo]{Corollary}
\newtheorem*{corollario}{Corollary}
\newtheorem*{weyl}{Weyl criterion}
\theoremstyle{definition}
\newtheorem{defi}[teo]{Definition}
\newtheorem{oss}[teo]{Remark}
\newtheorem*{ack}{Acknowledgments}

\title{Essential spectrum for the $p-$Laplacian}

\author[Brasco]{Lorenzo Brasco}
\address[L.\ Brasco]{Dipartimento di Matematica e Informatica
	\newline\indent
	Universit\`a degli Studi di Ferrara
	\newline\indent
	Via Machiavelli 35, 44121 Ferrara, Italy}
\email{lorenzo.brasco@unife.it}

\author[Briani]{Luca Briani}
\address[L.\ Briani]{School of Computation, Information and Technology		\newline\indent
 	Technical University of Munich
 	\newline\indent
	Boltzmannstra\ss e 3, 85748 Garching bei M\"unchen, Germany}
\email{luca.briani@tum.de}

\author[Franzina]{Giovanni Franzina}
\address[G.\ Franzina]{Consiglio Nazionale delle Ricerche
\newline\indent
Via dei Taurini 19, 00185 Roma, Italy}
\email{giovanni.franzina@cnr.it}

\date{\today}

\subjclass[2010]{35P30, 47J10, 49R05}
\keywords{Nonlinear eigenvalue problems, $p-$Laplacian, essential spectrum, Persson's Theorem.}

\begin{document}

\begin{abstract}
We introduce a variational notion of essential spectrum for the Dirichlet $p-$Laplacian. We then extend the classical Persson Theorem to this nonlinear setting. This result provides a geometric characterization of the bottom of the essential spectrum, in terms of the sharp $L^p$ Poincar\'e constant ``at infinity''. We also show that in the case $p=2$ our construction of the essential spectrum is perfectly consistent with the classical theory. Finally, as an example, we compute the full spectrum of the Dirichlet $p-$Laplacian on a rectilinear strip: it is purely essential, with no embedded eigenvalues. 
The arguments of the proofs are elementary and new already for the linear case $p=2$.
\end{abstract}

\maketitle

\begin{center}
\begin{minipage}{10cm}
\small
\tableofcontents
\end{minipage}
\end{center}

\section{Introduction}

\subsection{A classic: Persson's gateway}
On an open set $\Omega\subseteq\mathbb{R}^N$, we consider the Dirichlet-Laplacian, defined trough the first variation of the associated quadratic form
\begin{equation}
\label{quadratic}
\varphi\mapsto \int_\Omega |\nabla\varphi|^2\,dx,\qquad \text{for}\ \varphi\in W^{1,2}_0(\Omega).
\end{equation}
The symbol $W^{1,2}_0(\Omega)$ is the classical notation for the closure of $C^\infty_0(\Omega)$ in the usual Sobolev space
\[
W^{1,2}(\Omega)=\Big\{\varphi\in L^2(\Omega)\, :\, \nabla \varphi\in L^2(\Omega;\mathbb{R}^N)\Big\},
\]
endowed with the norm
\[
\|\varphi\|_{W^{1,2}(\Omega)}=\left(\int_\Omega |\varphi|^2\,dx+\int_\Omega |\nabla\varphi|^2\,dx\right)^\frac{1}{2}.
\]
We define the sharp Poincar\'e constant
\[
\lambda(\Omega)=\inf_{\varphi\in C^\infty_0(\Omega)} \left\{\int_\Omega |\nabla \varphi|^2\,dx\,:\, \int_\Omega |\varphi|^2\,dx=1\right\}.
\]
A well-known result by Persson asserts that if we define the sharp Poincar\'e constant ``at infinity'' through
\[
\mathcal{E}(\Omega):=\sup_{R>0} \lambda(\Omega\setminus \overline{B_R}),\qquad \text{where}\ B_R=\{x\in\mathbb{R}^N\, :\, |x|<R\},
\]
then it holds
\[
\mathcal{E}(\Omega)=\inf \mathfrak{S}_{\rm ess}(\Omega),
\]
i.e. the constant $\mathcal{E}(\Omega)$ coincides with the infimum of the {\it essential spectrum} $\mathfrak{S}_{\rm ess}(\Omega)$ of the Dirichlet-Laplacian on $\Omega$, see \cite[Theorem 2.1]{Pe} or \cite[Theorem 14.11]{HS}. Persson's result provides a ``geometric'' characterization of the ``gateway'' to the essential spectrum of the Dirichlet-Laplacian on $\Omega$. It is part of his assertion that
\[
\mathcal{E}(\Omega)=+\infty \qquad \Longleftrightarrow\qquad \mathfrak{S}_{\rm ess}(\Omega)=\emptyset.
\]
When this happens, the spectrum of the Dirichlet-Laplacian on $\Omega$ is discrete, only made of eigenvalues with finite multiplicities accumulating to infinity. This in turn is equivalent to the compactness of the embedding
\[
W^{1,2}_0(\Omega)\hookrightarrow L^2(\Omega),
\]
see for example \cite[Theorem 10.1.5]{BS}. We refer the reader to \cite[Chapter 15]{Maz} for a characterization of open sets supporting such a compact embedding.
\vskip.2cm\noindent
To set the stage, let us recall what the essential spectrum is. For our purposes, its definition is not necessary: 
it is sufficient to rely on the following well-known equivalence result, sometimes called {\it Weyl criterion} (see for example \cite[Theorem 9.1.2]{BS}, \cite[Theorem 7.2]{HS} or \cite[Lemma 6.17]{Te}).
 \begin{weyl}
 $\lambda\in \mathfrak{S}_{\rm ess}(\Omega)$ if and only if there exists 
 $\{u_n\}_{n\in\mathbb{N}}\subseteq L^2(\Omega)$ such that:
\begin{enumerate}
\item[(W1)] $\|u_n\|_{L^2(\Omega)}=1$ for every $n\in\mathbb{N}$;
\vskip.2cm
\item[(W2)] $\{u_n\}_{n\in\mathbb{N}}$ weakly converges to $0$ in $L^2(\Omega)$, as $n$ goes to $\infty$;
\vskip.2cm
\item[(W3)] each $u_n$ belongs to the domain of the operator, i.e. $u_n\in W^{1,2}_0(\Omega)$ and $-\Delta u_n\in L^2(\Omega)$;
\vskip.2cm
\item[(W4)] we have
\[
\lim_{n\to\infty} \|-\Delta u_n-\lambda\,u_n\|_{L^2(\Omega)}=0.
\]
\end{enumerate}
Such a sequence is also called {\rm singular Weyl sequence at the level $\lambda$}. 
 \end{weyl}
 \begin{oss}
In the framework of Critical Point Theory, a sequence having properties (W1), (W2) and (W4) resembles very much a non-compact Palais-Smale sequence for the quadratic form \eqref{quadratic}, constrained to the set 
 \[
\Big\{u\in W^{1,2}_0(\Omega)\, :\, \|u\|_{L^2(\Omega)}=1\Big\},
\]
see for example \cite[Chapter II, Section 2]{St}. In such a context, the property (W4) is more naturally formulated as
\[
\lim_{n\to\infty} \|-\Delta u_n-\lambda\,u_n\|_{W^{-1,2}(\Omega)}=0,
\]
where $W^{-1,2}(\Omega)$ is the topological dual space of the form domain $W^{1,2}_0(\Omega)$. Note that $-\Delta u_n$ automatically belongs to the dual space $W^{-1,2}(\Omega)$, when $u_n\in W^{1,2}_0(\Omega)$. We will come back on this similarity in a moment.
\end{oss}
\subsection{Goal of the paper}
From the Calculus of Variations standpoint, the curious spectral investigator could be tempted to replace the quadratic form \eqref{quadratic}
by the more general functional (here $1<p<\infty$)
\begin{equation}
\label{functional}
\varphi\mapsto \int_\Omega |\nabla\varphi|^p\,dx,\qquad \text{for}\ \varphi\in W^{1,p}_0(\Omega).
\end{equation}
The first variation of this functional gives rise to a nonlinear differential operator, the $p-$Laplacian with homogeneous Dirichlet boundary conditions on $\partial\Omega$. This is the operator defined in weak form by
\[
\langle -\Delta_p u,\varphi\rangle:=\int_\Omega \langle |\nabla u|^{p-2}\,\nabla u,\nabla\varphi\rangle\,dx,\qquad \text{for every}\ \varphi\in C^\infty_0(\Omega).
\]
There is a well-known and well-studied {\it eigenvalue problem} associated to this operator. This goes as follows.
\begin{defi}
Let $1<p<\infty$ and let $\Omega\subseteq\mathbb{R}^N$ be an open set. We say that $\lambda\in\mathbb{R}$ is an {\it eigenvalue of the Dirichlet $p-$Laplacian on $\Omega$} if there exists $u\in W^{1,p}_0(\Omega)\setminus\{0\}$ such that
\[
-\Delta_p u=\lambda\,|u|^{p-2}\,u,\qquad \text{in}\ \Omega.
\]
This has to be intended in weak sense, i.e.
\[
\int_\Omega \langle |\nabla u|^{p-2}\,\nabla u,\nabla \varphi\rangle\,dx=\lambda\,\int_\Omega |u|^{p-2}\,u\,\varphi\,dx,\qquad \text{for every}\ \varphi\in W^{1,p}_0(\Omega).
\]
We denote the collection of all eigenvalues on $\Omega$ by the symbol $\mathfrak{S}_{{\rm eigen},p}(\Omega)$.
\end{defi}
In other words, in view of the Lagrange Multipliers Rule, these eigenvalues correspond to the critical values of the functional \eqref{functional}
constrained to the set
\[
\Big\{u\in W^{1,p}_0(\Omega)\, :\, \|u\|_{L^p(\Omega)}=1\Big\}.
\]
For $p=2$, we exactly get back the eigenvalues of the Dirichlet-Laplacian. Observe that if we set
\[
\lambda_p(\Omega)=\inf_{\varphi\in C^\infty_0(\Omega)} \left\{\int_\Omega |\nabla \varphi|^p\,:\, \int_\Omega |\varphi|^p\,dx=1\right\},
\]
then we clearly have
\[
\lambda\ge \lambda_p(\Omega),\qquad \text{for every}\ \lambda\in\mathfrak{S}_{{\rm eigen},p}(\Omega).
\]
As a small historical comment, we recall that one of the very first appearance of this eigenvalue problem can be traced back to Lieb's paper \cite{Li}. Since then, it has been extensively studied in the literature, at least in the case where the embedding
\[
W^{1,p}_0(\Omega)\hookrightarrow L^p(\Omega),
\]
is compact. We refer for example to the classical papers \cite{An, CDeFG, Fri, GP, Le, Szu} or to the monographs \cite{Fra, Lin} and the references therein contained.
\vskip.2cm\noindent
On the contrary, in the non-compact case, the situation has not been vey much investigated. In particular, it is not clear (and hopeless, in principle) whether one can go beyond the simple situation of eigenvalues and construct a complete Spectral Theory.
In view of the previous subsection, a couple of naive (yet very natural) question may arise at this point:
\begin{itemize}
\item is it possible to give a definition of essential spectrum for the Dirichlet $p-$Laplacian, when $p\not=2$? 
\vskip.2cm
\item if {\it yes}, is it still possible to get a Persson--type result for the infimum of the essential spectrum? Does it coincide with 
\[
\mathcal{E}_p(\Omega):=\sup_{R>0} \lambda_p(\Omega\setminus \overline{B_R})\,?
\]
\end{itemize}
Answering these questions is the main goal of this paper.
\begin{oss} 
In the recent paper \cite{BraBriPri_low} it has been proved that it is possible to construct eigenvalues below the threshold $\mathcal{E}_p(\Omega)$, even in the non-compact case, by means of minmax formulations {\it \`a la} Courant-Fischer. We refer to \cite[Main Theorem]{BraBriPri_low} for the precise statement and some further comments. This result parallels a similar phenomenon valid for the case $p=2$, thus suggesting that $\mathcal{E}_p(\Omega)$ should be the natural candidate for the infimum of the essential spectrum for $p\not=2$....provided we can define it!
\end{oss}
\subsection{Main results}

We need to recall the definition of constrained Palais-Smale sequences: guided by the Weyl Criterion previously recalled, we will distinguish between {\it regular} and {\it singular} ones.
\begin{defi}
\label{defi:PS}
Let $1<p<\infty$ and let $\Omega\subseteq\mathbb{R}^N$ be an open set. We say that $\{u_n\}_{n\in\mathbb{N}}\subseteq W^{1,p}_0(\Omega)$ is a {\it constrained Palais-Smale sequence at the level $\lambda$} if:
\vskip.2cm
\begin{enumerate}
\item[(i)] $\|u_n\|_{L^p(\Omega)}=1$, for every $n\in\mathbb{N}$;
\vskip.2cm
\item[(ii)] $\lim\limits_{n\to\infty} \displaystyle\int_\Omega |\nabla u_n|^p\,dx=\lambda$;
\vskip.2cm
\item[(iii)] we have 
\[
\lim_{n\to\infty} \Big\|-\Delta_p u_n-\lambda\,|u_n|^{p-2}\,u_n\Big\|_{W^{-1,p'}(\Omega)}=0,
\]
where $W^{-1,p'}(\Omega)$ is the topological dual space of $W^{1,p}_0(\Omega)$.
\end{enumerate}
We say that such a sequence is:
\begin{itemize}
\item {\it regular}, if it admits a subsequence $\{u_{n_k}\}_{k\in\mathbb{N}}$ such that 
\[
\lim_{k\to\infty} \|u_{n_k}-u\|_{L^p(\Omega\cap B_R)}=0,\qquad \text{for every}\ R>0,
\]
for some $u\in W^{1,p}_0(\Omega)\setminus\{0\}$;
\vskip.2cm
\item {\it singular}, if we have
\[
\lim_{n\to\infty}  \|u_n\|_{L^p(\Omega\cap B_R)}=0,\qquad \text{for every}\ R>0.
\]
\end{itemize}
\end{defi}
\begin{oss}
It is not difficult to see that, for a constrained Palais-Smale sequence, one has the alternative: either it is regular or it is singular (see Proposition \ref{prop:urinason}). A level $\lambda$ admitting constrained Palais-Smale sequences is thus forced to arise from at least one of the two approximation procedures. 
\end{oss}
We are now ready for the most important definition of the paper: that of essential spectrum of the $p-$Laplacian. We will show that for $p=2$ our definition boils down to the usual one, characterized by means of the Weyl criterion recalled above. Thus, our analysis is a genuine extension to the nonlinear case of the classical case. This important fact will be shown in Theorem \ref{teo:essenza2} below.
\begin{defi}
Let $1<p<\infty$ and let $\Omega\subseteq\mathbb{R}^N$ be an open set. We define the {\it variational essential spectrum of the Dirichlet $p-$Laplacian on $\Omega$} as the the following set
\[
\mathfrak{S}_{{\rm ess},p}(\Omega)=\Big\{\lambda\, :\, \exists\ \text{singular constrained Palais-Smale sequence at the level $\lambda$} \Big\}.
\]
Finally, we set
\[
\mathfrak{S}_{p}(\Omega)=\mathfrak{S}_{{\rm ess},p}(\Omega)\cup \mathfrak{S}_{{\rm eigen},p}(\Omega),
\]
and call it the {\it variational spectrum of the Dirichlet $p-$Laplacian on $\Omega$}.
\end{defi}
\begin{oss}
We observe that $\mathfrak{S}_{{\rm ess},p}(\Omega)$ and $\mathfrak{S}_{{\rm eigen},p}(\Omega)$ are not necessarily disjoint sets. In general, this is not true already for the linear case $p=2$ (see for example \cite{Wi} and the references therein). 
\end{oss}

We can now present the main achievement of this paper: a generalization of the celebrated Persson Theorem to the case of the $p-$Laplacian. We will use the convention that
\[
\inf \emptyset=+\infty.
\] 
\begin{main}[Nonlinear Persson Theorem]
Let $1<p<\infty$ and let $\Omega\subseteq\mathbb{R}^N$ be an open set. We set 
\[
\mathcal{E}_p(\Omega)=\sup_{R>0} \lambda_p(\Omega\setminus\overline{B_R}).
\]
Then we have
\[
\mathcal{E}_p(\Omega)=\inf \mathfrak{S}_{{\rm ess},p}(\Omega).
\]
Moreover, if $\mathcal{E}_p(\Omega)<+\infty$ then it belongs to $\mathfrak{S}_{{\rm ess},p}(\Omega)$ and thus it coincides with its minimum.
\end{main}
The proof of this result is postponed to Section \ref{sec:3}. 
\vskip.2cm\noindent
We immediately derive an interesting consequence of the previous result: the sharp Poincar\'e constant $\lambda_p$ always coincides with the bottom of the spectrum, exactly as in the linear case (see for example \cite[Chapter 10, Section 1.1]{BS}). 
\begin{corollario}
Let $1<p<\infty$ and let $\Omega\subseteq\mathbb{R}^N$ be a non-empty open set. Then we have 
\[
\lambda_p(\Omega)=\min\mathfrak{S}_{p}(\Omega).
\]
\end{corollario}
\begin{proof}
By construction and the definition of $\lambda_p(\Omega)$, it is easily seen that 
\[
\lambda_p(\Omega)\le \lambda,\qquad \text{for every}\ \lambda\in\mathfrak{S}_{p}(\Omega).
\]
To prove that $\lambda_p(\Omega)$ always belongs to the spectrum, we may distinguish two cases: 
\begin{itemize}
\item if $\lambda_p(\Omega)<\mathcal{E}_p(\Omega)$, then we have
\[
\lambda_p(\Omega)\in \mathfrak{S}_{{\rm eigen},p}(\Omega)\subseteq \mathfrak{S}_{p}(\Omega),
\]
thanks to \cite[Main Theorem]{BraBriPri_low};
\vskip.2cm
\item if $\lambda_p(\Omega)=\mathcal{E}_p(\Omega)$, then in particular $\mathcal{E}_p(\Omega)<+\infty$ and by the nonlinear Persson Theorem we have 
\[
\lambda_p(\Omega)=\min \mathfrak{S}_{{\rm ess},p}(\Omega).
\]
\end{itemize}
This concludes the proof.
\end{proof}

\subsection{Some comments on the proof}
As in Persson's paper \cite{Pe}, the Main Theorem will be obtained by proving separately that
\[
\mathcal{E}_p(\Omega)\le \inf\mathfrak{S}_{{\rm ess},p}(\Omega)\qquad \text{and}\qquad \mathcal{E}_p(\Omega)\ge \inf\mathfrak{S}_{{\rm ess},p}(\Omega).
\]
While the leftmost inequality is relatively easy to be proven, the rightmost one is much more challenging (unless $\mathcal{E}_p(\Omega)=+\infty$, of course). Indeed, we need to construct a singular constrained Palais-Smale sequence $\{u_n\}_{n\in\mathbb{N}}$ at the level $\mathcal{E}_p(\Omega)$. The difficult point is to obtain that such a sequence is made of ``almost critical points'', i.e. that
\[
\lim_{n\to\infty} \Big\|-\Delta_p u_n-\mathcal{E}_p(\Omega)\,|u_n|^{p-2}\,u_n\Big\|_{W^{-1,p'}(\Omega)}=0.
\]
For this part, the proof by Persson exploits the Spectral Resolution of the Dirichlet-Laplacian (see \cite[Lemma 2.1]{Pe}). This is nothing but the Spectral Theorem in its general formulation (see for example \cite[Theorem 3.6]{Te}), which permits us to write the Dirichlet-Laplacian as a linear superposition of orthogonal projection operators. This approach is not feasible in the nonlinear case and a new strategy is needed.
\par
The idea we use is quite natural: in view of its definition, we can think of $\mathcal{E}_p(\Omega)$ as the first eigenvalue of a set ``escaping at infinity''. We then try to consider the ground states of the Schr\"odinger--type operators 
\[
\varphi\mapsto-\Delta_p \varphi + V_n\,|\varphi|^{p-2}\,\varphi,\qquad \text{for}\ n\in\mathbb{N},
\]
for a sequence of ``repulsive'' potentials $\{V_n\}_{n\in\mathbb{N}}$. The effect of the potential term $V_n$ is that of trying to concentrate the ground states outside the set $\Omega\cap B_n$. For example, one can take $V_n$ to be a bump-function supported on $B_n$ and multiplied by a factor $\Lambda_n$ diverging with $n\in\mathbb{N}$. However, in general these operators will not have any ground states, due to the lack of coercivity of the associated energy functional
\[
\int_\Omega |\nabla \varphi|^p\,dx+\int_\Omega V_n\,|\varphi|^p\,dx.
\]
In order to overcome this problem, we ``artificially'' create a bit of compactness, by adding a further potential term $W_n$: this time, this is a strongly confining potential, but with a positive intensity factor $\varepsilon_n$ which vanishes as $n$ diverges. For example, one could imagine to add
\[
W_n(x)=\varepsilon_n\,|x|,\qquad \text{for}\ x\in\Omega.
\]
Thus, we are led to consider the following ground state energy
\[
\inf_{\varphi\in W^{1,p}_0(\Omega)} \left\{\int_\Omega |\nabla \varphi|^p\,dx+\int_\Omega V_n\,|\varphi|^p\,dx+\varepsilon_n\,\int_\Omega |x|\,|\varphi|^p\,dx:\, \|\varphi\|_{L^p(\Omega)}=1\right\}.
\]
It is not difficult to see that such an infimum is attained by a function $u_n$ (actually, even uniquely, up to the choice of the sign).
It turns out that, by carefully choosing the relation between the ``repulsive effect'' of $V_n$ and the ``confining strength'' $\varepsilon_n$, the sequence $u_n$ has the desired properties. In particular, it is made of ``almost critical points'' at the level $\mathcal{E}_p(\Omega)$ and it locally converges to $0$, due to the effect of the factor $V_n$.
\par
We point out that this proof, partly inspired by ideas recently used by the first two authors in \cite{BraBriPri_low, BraBriPri_periodic} and \cite{BraBriPri_steiner}, is entirely based on variational principles and thus nonlinear in nature. In particular, it gives a new proof of Persson's result, not relying on the Spectral Theorem.

\subsection{Bonus track: an example}
As an application of the theory developed in this paper, we compute in Theorem \ref{teo:striscia} the variational spectrum of the Dirichlet $p-$Laplacian on a rectilinear strip
\[
\mathcal{S}_\alpha=(-\alpha,\alpha)\times\mathbb{R}.
\]
We show that it has the following structure
\[
\mathfrak{S}_{\mathrm{ess}, p}(\mathcal{S}_\alpha) = \left[\lambda_p\big((-\alpha,\alpha)\big), +\infty\right)\qquad \text{and}\qquad \mathfrak{S}_{{\rm eigen},p}(\mathcal{S}_\alpha)=\emptyset,
\]
i.e. the spectrum is purely essential, coinciding with the half-line starting at the sharp Poincar\'e constant of the horizontal section. Moreover, there are no eigenvalues, not even embedded ones.
This result generalizes what happens for the case $p=2$ (see for example \cite[Introduction]{Wi} and the references therein). 
\par
In order to compute the essential spectrum, we construct singular constrained Palais-Smale sequences $\{u_n\}_{n\in\mathbb{N}}$ arguing as follows: 
\begin{itemize}
\item choose $\lambda> \lambda_p((-\alpha,\alpha))$ and take the first eigenfunction of the rectangle $(-\alpha,\alpha)\times(-\ell,\ell)$, with $\ell>0$ chosen so that its first eigenvalue coincides with $\lambda$;
\vskip.2cm
\item extend this eigenfunction by $n$ odd reflections and then set this extension equal to $0$ in the remainder of $\mathcal{S}_\alpha$. We call $\psi_n$ the function obtained in this way;
\vskip.2cm
\item define $u_n=\psi_n/\|\psi_n\|_{L^p}$  and then compute $-\Delta_p u_n-\lambda\, |u_n|^{p-2}\,u_n$;
\vskip.2cm
\item show that the extension by zero has created a singular term in $-\Delta_p u_n$, which is however infinitesimal in $W^{-1,p'}$ norm.
\end{itemize} 
This strategy suitably extends the argument used in \cite[Proposition 6.1]{BiaBraOgn} for $p=2$. We show here that it is possible to safely perform this construction, without having an explicit expression of this first eigenfunction of a rectangle.
\par
On the other hand, proving the absence of eigenvalues embedded in the essential spectrum in general is a very difficult and interesting task, already in the linear case. For $p=2$, a possible strategy is to combine {\it Carleman estimates} and the {\it Unique Continuation Principle}, as in the classical paper \cite{Ro} by Roze (see also the more recent one \cite{KT}, for example). Such a strategy would be hopeless for $p\not=2$, since the validity of the UCP is still a major open problem for the $p-$Laplacian. 
\par
Another possibility, which is fruitful in the case of peculiar geometries, is to use integral identities of Pohozaev-Rellich--type. We proceed in this way, by proving an {\it ad hoc} integral identity of this type (see Proposition \ref{prop:pocozaev} below), inspired by the papers \cite{EL} and \cite{Re}. As it is typical in the case of the $p-$Laplacian, a careful approximation argument will be needed, in order to circumvent the possible lack of regularity of eigenfunctions.

\subsection{Plan of the paper}
We start with Section \ref{sec:2}, discussing the behaviour of the gradients for a constrained Palais-Smale sequence. In the same section, we also show that levels $\lambda$ for which there exists a regular constrained Palais-Smale sequence are actually eigenvalues, as one should expect. The proof of the Main Theorem is contained in Section \ref{sec:3}.
In Section \ref{sec:4} we show that for $p=2$ our definition of essential spectrum boils down to the usual one. The (long) Section \ref{sec:5} contains the exact determination of the spectrum of the Dirichlet $p-$Laplacian in a planar strip. Finally, Appendix \ref{sec:A} contains a technical result, suitably adapted from the recent paper \cite{BraBriPri_periodic}, which is crucial for constructing a singular Palais-Smale sequence at the ``gateway'' level $\mathcal{E}_p(\Omega)$.

\begin{ack}
We wish to thank Francesca Bianchi, Roberto Ognibene and Francesca Prinari for some previous collaborations, which lead us to a better understanding of the problem. L.\, Brasco thanks Vladimir Bobkov, who generously shared a preliminary version of his paper \cite{Bob}. The results of this paper have been presented during a talk of G.\, Franzina in Munich in April 2026: he wishes to thank the Department of Mathematics of the Technical University of Munich for the kind invitation and the nice atmosphere during the staying.
\par
L.\,Briani and G. Franzina are members of the {\it Gruppo Nazionale per l'Analisi Matematica, la Probabilit\`a
e le loro Applicazioni} (GNAMPA) of the Istituto Nazionale di Alta Matematica (INdAM). 
\par
L.\,Brasco has been financially supported by the {\it Fondo di Ateneo per la Ricerca} FAR 2024 and the {\it Fondo per l'Incentivazione alla Ricerca Dipartimentale} FIRD 2024 of the University of Ferrara.
\par
L.\, Briani gratefully acknowledges the financial support of the project GNAMPA 2026  ``Problemi di ottimizzazione di forma in contesti anisotropi e non-locali" ({\tt  CUP E53C25002010001}) and of the DFG through the Emmy Noether Programme (project number 509436910)
\par
G.\, Franzina gratefully acknowledges the financial support of the project GNAMPA 2026 ``Variational and PDE methods in the study of topological singularities in complex materials” ({\tt CUP E53C2500201000}).

\end{ack}

\section{Preliminaries}
\label{sec:2}
We start with the following basic result.
\begin{lm}
\label{lm:mizuno}
Let $1<p<\infty$ and let $\Omega\subseteq\mathbb{R}^N$ be an open set. Let 
$\{u_n\}_{n\in\mathbb{N}}\subseteq W^{1,p}_0(\Omega)$ be a sequence weakly converging in $W^{1,p}(\Omega)$ to a function $u\in W^{1,p}_0(\Omega)$. Then we have
\begin{equation}
\label{standard}
\lim_{n\to\infty}\|u_n-u\|_{L^{p}(\Omega\cap B_R)}=0,\quad  \text{for every}\ R>0.
\end{equation}
\end{lm}
\begin{proof}
The proof is quite standard, but we detail the argument for completeness.
As usual when working with the space $W^{1,p}_0(\Omega)$, we can always think of our functions $\{u_n\}_{n\in\mathbb{N}}$ and $u$ as defined on the whole $\mathbb{R}^N$: we extend them to be $0$ outside $\Omega$. Thus, we have in particular that $\{u_n\}_{n\in\mathbb{N}}\subseteq W^{1,p}(B_R)$, for every $R>0$.
Observe that for every $R>0$, the embedding $W^{1,p}(B_R)\hookrightarrow L^p(B_R)$ is compact. By using this fact and the weak convergence in $W^{1,p}(\Omega)$ (which holds by assumption), a standard argument (see for example \cite[Lemma 3.8.7 \& Remark 3.9.5]{BraBook}) also gives \eqref{standard}, without the necessity of extracting further subsequences.
\end{proof}
The next result is quite important for the whole discussion.
\begin{prop}
\label{prop:urinason}
Let $1<p<\infty$ and let $\Omega\subseteq\mathbb{R}^N$ be an open set. Let $\{u_n\}_{n\in\mathbb{N}}\subseteq W^{1,p}_0(\Omega)$ be a constrained Palais-Smale sequence, at the level $\lambda$. Then $\{u_n\}_{n\in\mathbb{N}}$ is either regular or singular, according to Definition \ref{defi:PS}.
\end{prop}
\begin{proof}
We first observe that, if we define the distance
\[
d_{L^p}(u,v):=\sum_{j=1}^\infty \frac{1}{j^2}\, \frac{\|u-v\|_{L^p(\Omega\cap B_j)}}{1+\|u-v\|_{L^p(\Omega\cap B_j)}},\qquad \text{for every}\ u,v\in L^p(\Omega),
\]
we have that
\[
\lim_{n\to\infty} \|\varphi_n-\varphi\|_{L^p(\Omega\cap B_R)}=0,\quad \text{for every}\ R>0\qquad \Longleftrightarrow\qquad \lim_{n\to\infty} d_{L^p}(\varphi_n,\varphi)=0.
\] 
Let us suppose that the sequence $\{u_n\}_{n\in\mathbb{N}}$ of the statement is not singular. Accordingly, we have that 
\[
\delta:=\limsup_{n\to\infty} d_{L^p}(u_n,0)>0.
\]
We can extract a subsequence $\{u_{n_k}\}_{k\in\mathbb{N}}$ such that 
\[
\lim_{k\to\infty} d_{L^p}(u_{n_k},0)=\limsup_{n\to\infty} d_{L^p}(u_n,0)=\delta>0.
\]
By definition of constrained Palais-Smale sequence, we get in particular that $\{u_{n_k}\}_{k\in\mathbb{N}}$ is bounded in $W^{1,p}_0(\Omega)$. Up to a further subsequence, we thus get that $\{u_{n_k}\}_{k\in\mathbb{N}}$ converges weakly in $W^{1,p}(\Omega)$ to a function $u\in W^{1,p}_0(\Omega)$. We argue by contradiction and suppose that $u\equiv 0$: by Lemma \ref{lm:mizuno}, we would get that 
\[
\lim_{k\to\infty} \|u_{n_k}\|_{L^p(\Omega\cap B_R)}=0,\qquad \text{for every}\ R>0.
\]
In particular, by choosing $R=j\in \mathbb{N}\setminus\{0\}$, we would obtain
\[
\delta=\lim_{k\to\infty} d_{L^p}(u_{n_k},0)=0,
\] 
a contradiction.
\end{proof}

The following compactness result will be useful. This is a particular case of \cite[Lemma A.1]{BraBriPri_periodic}, to which we refer for the elementary proof.
\begin{lm}
\label{lm:lemmadelca}
Let $1<p<\infty$ and let $\Omega\subseteq\mathbb{R}^N$ be an open set such that $\mathcal{E}_{p}(\Omega)>0$. Let 
$\{u_n\}_{n\in\mathbb{N}}\subseteq W^{1,p}_0(\Omega)$
be a sequence with the following properties:
\begin{itemize}
\item $\|u_n\|_{L^p(\Omega)}=1$, for every $n\in\mathbb{N}$;
\vskip.2cm
\item $\{u_n\}_{n\in\mathbb{N}}$ weakly converges in $W^{1,p}(\Omega)$ to some function $u\in W^{1,p}_0(\Omega)$;
\vskip.2cm
\item there exist $\lambda<\mathcal{E}_{p}(\Omega)$ and $n_0\in\mathbb{N}$ such that 
\[
\int_\Omega |\nabla u_n|^p\,dx\le \lambda,\qquad \text{for every}\ n\ge n_0.
\]
\end{itemize}
Then, there exists $0<\mathfrak{d}=\mathfrak{d}(\lambda/\mathcal{E}_{p}(\Omega))<1$ such that $\|u\|_{L^p(\Omega)}\ge \mathfrak{d}$. Moreover, the constant $\mathfrak{d}$ is such that we have 
\[
\lim_{\mathcal{E}_{p}(\Omega)\to+\infty} \mathfrak{d}=1.
\]
\end{lm}
The next pair of results concern the behavior of the gradients for a constrained Palais-Smale sequence. We start from the singular case.
\begin{prop}
\label{prop:1}
Let $1<p<\infty$ and let $\Omega\subseteq\mathbb{R}^N$ be an open set. Let $\{u_n\}_{n\in\mathbb{N}}\subseteq W^{1,p}_0(\Omega)$ be a singular constrained Palais-Smale sequence at the level $\lambda$. Then we have
\[
\lim_{n\to\infty} \|\nabla u_n\|_{L^p(\Omega\cap B_R)}=0,\qquad \text{for every}\ R>0,
\]
as well.
\end{prop}
\begin{proof}
By definition, we already know that 
\begin{equation}
\label{Lp}
\lim_{n\to\infty} \|u_n\|_{L^p(\Omega\cap B_R)}=0,\qquad \text{for every}\ R>0.
\end{equation}
We need to prove the same property for the sequence $\{\nabla u_n\}_{n\in\mathbb{N}}$. We fix $R>0$ and we take $\eta\in C^\infty_0(B_{2R})$ such that 
\[
0\le \eta\le 1,\qquad \eta\equiv 1\ \text{on}\ B_R,\qquad |\nabla \eta|\le \frac{C}{R},
\]
for some $C>0$.
We thus obtain
\[
\begin{split}
\int_{\Omega\cap B_R} |\nabla u_n|^p\,dx\le \int_{\Omega} |\nabla u_n|^p\,\eta\,dx&=\int_{\Omega} \langle |\nabla u_n|^{p-2}\,\nabla u_n,\nabla(u_n\,\eta)\rangle\,dx\\
&-\int_{\Omega} \langle |\nabla u_n|^{p-2}\,\nabla u_n,\nabla\eta\rangle\,u_n\,dx\\
&=\left\langle-\Delta_p u_n-\lambda\,|u_n|^{p-2}\,u_n,u_n\,\eta\right\rangle_{(W^{-1,p'}(\Omega),W^{1,p}_0(\Omega))}\\
&-\int_{\Omega} \langle |\nabla u_n|^{p-2}\,\nabla u_n,\nabla\eta\rangle\,u_n\,dx\\
&+\lambda\,\int_\Omega |u_n|^p\,\eta\,dx.
\end{split}
\]
In particular, we get
\begin{equation}
\label{camado}
\begin{split}
\int_{\Omega\cap B_R} |\nabla u_n|^p\,dx&\le \Big\|-\Delta_p u_n-\lambda\,|u_n|^{p-2}\,u_n\Big\|_{W^{-1,p'}(\Omega)}\,\|u_n\,\eta\|_{W^{1,p}(\Omega)}\\
&+\frac{C}{R}\,\|\nabla u_n\|_{L^p(\Omega)}^{p-1}\,\|u_n\|_{L^p(\Omega\cap B_{2R})}+\lambda\,\|u_n\|_{L^p(\Omega\cap B_{2R})}^p. 
\end{split}
\end{equation}
Observe that
\[
\begin{split}
\|u_n\,\eta\|_{W^{1,p}(\Omega)}&\le \|u_n\,\eta\|_{L^p(\Omega)}+\|u_n\,\nabla\eta+\eta\,\nabla u_n\|_{L^p(\Omega)}\\
&\le 1+\frac{C}{R}+\|\nabla u_n\|_{L^p(\Omega)},
\end{split}
\]
which is uniformly bounded in $n$, by assumption. By using this fact, \eqref{Lp} and the definition of Palais-Smale sequence, we get the conclusion from \eqref{camado}.
\end{proof}
In the regular case, we have a similar result. This reads as follows.
\begin{prop}
\label{prop:2}
Let $1<p<\infty$ and let $\Omega\subseteq\mathbb{R}^N$ be an open set. Let $\{u_n\}_{n\in\mathbb{N}}\subseteq W^{1,p}_0(\Omega)$ be a regular constrained Palais-Smale sequence at the level $\lambda$. Let $\{u_{n_k}\}_{k\in\mathbb{N}}\subseteq \{u_n\}_{n\in\mathbb{N}}$ and $u\in W^{1,p}_0(\Omega)\setminus\{0\}$ be such that 
\[
\lim_{k\to\infty} \|u_{n_k}-u\|_{L^p(\Omega\cap B_R)}=0,\qquad \text{for every}\ R>0,
\]
then we also have
\[
\lim_{k\to\infty} \|\nabla u_{n_k}-\nabla u\|_{L^p(\Omega\cap B_R)}=0,\qquad \text{for every}\ R>0.
\]
\end{prop}
\begin{proof}
We preliminary prove that we have weak convergence of the gradients. Indeed, 
observe that $\{u_{n_k}\}_{k\in\mathbb{N}}$ is in particular a bounded sequence in $W^{1,p}_0(\Omega)$. Thus, it converges weakly in $W^{1,p}(\Omega)$, up to a further subsequence. Let us call $v$ the weak limit and observe that $v\in W^{1,p}_0(\Omega)$, since the latter is a weakly closed space. Let us take $\phi\in C^\infty_0(\Omega;\mathbb{R}^N)$, in particular we also have that $\phi\in C^{\infty}_0(\Omega\cap B_R)$, for $R>0$ large enough (depending on $\phi$). 
Thus, by using the weak convergence in $W^{1,p}(\Omega)$, the definition of weak gradient and the strong local convergence in $L^p$, we have
\[
\int_\Omega \langle \nabla v,\phi\rangle\,dx=\lim_{k\to\infty} \int_\Omega \langle \nabla u_{n_k},\phi\rangle\,dx=-\lim_{k\to\infty} \int_\Omega u_{n_k}\,\mathrm{div}\phi\,dx=-\int_\Omega u\,\mathrm{div}\phi\,dx.
\]
This implies that 
\[
\nabla v=\nabla u,\qquad \text{a.\,e. in}\ \Omega.
\]
Since $v-u\in W^{1,p}_0(\Omega)$, the previous property implies that $v=u$ (see for example \cite[Proposition 3.7.11]{BraBook}). We have thus identified the weak limit. 
\par
Observe that the previous argument can be repeated for any subsequence, always getting the given function $u$ in the limit. By a standard argument, this shows that we actually have weak convergence of the initial subsequence $\{\nabla u_{n_k}\}_{k\in\mathbb{N}}$, without the necessity to pass to a further subsequence.
\vskip.2cm\noindent
We now come to the proof of the local strong $L^p$ convergence of the gradients. Let us set, for ease of notation
\[
\theta_{k}:=\Big\|-\Delta_p u_{n_k}-\lambda\,|u_{n_k}|^{p-2}\,u_{n_k}\Big\|_{W^{-1,p'}(\Omega)}.
\]
Also, for every  $R>0$, we take $\eta\in C^{\infty}_0(B_{2R})$ the same cut-off function of the previous proof. By using that $(u_{n_k}-u)\,\eta\in W^{1,p}_0(\Omega)$ we can write
\[
\begin{split}
\int_\Omega \langle |\nabla u_{n_k}|^{p-2}\,\nabla u_{n_k},\nabla (u_{n_k}-u)\rangle\,\eta\,dx&=\left\langle -\Delta_p u_{n_k}-\lambda\,|u_{n_k}|^{p-2}\,u_{n_k},\, (u_{n_k}-u)\,\eta\right\rangle_{(W^{-1,p'}(\Omega), W^{1,p}_0(\Omega))}\\
&-\lambda\,\int_\Omega |u_{n_k}|^{p-2}\,u_{n_k}\,(u_{n_k}-u)\,\eta\,dx\\
&-\int_\Omega \langle |\nabla u_{n_k}|^{p-2}\,\nabla u_{n_k},\nabla \eta\rangle\,(u_{n_k}-u)\,dx.
\end{split}
\]
Each term on the right-hand side of the previous inequality is comfortably estimated using H\"older inequality and the properties of $\eta$. Namely, we have
\[
\left|\left\langle -\Delta_p u_{n_k}-\lambda\,|u_{n_k}|^{p-2}\,u_{n_k},\, (u_{n_k}-u)\,\eta\right\rangle_{(W^{-1,p'}(\Omega), W^{1,p}_0(\Omega))}\right|\le  \theta_k \, \|\eta\, (u_{n_k}-u)\|_{W^{1,p}(\Omega)},
\]
\[
\left|\lambda\,\int_\Omega |u_{n_k}|^{p-2}\,u_{n_k}\,(u_{n_k}-u)\,\eta\, dx\right|\le  \|u_{n_k}-u\|_{L^{p}(\Omega\cap B_{2R})},
\]
and
\[
\left|\int_\Omega \langle |\nabla u_{n_k}|^{p-2}\,\nabla u_{n_k},\nabla \eta\rangle\,(u_{n_k}-u)\,dx\right|\le \frac{C}{R}\, \|\nabla u_{n_k}\|^{p-1}_{L^p(\Omega)}\,\|u_{n_k}- u\|_{L^{p}(\Omega\cap B_{2R})}.
\]
In particular, thanks to the strong local convergence in $L^p$ and being $\{\theta_k\}_{k\in\mathbb{N}}$ infinitesimal, we obtain that
\[
\lim_{k\to\infty}\int_\Omega \langle |\nabla u_{n_k}|^{p-2}\,\nabla u_{n_k},\nabla (u_{n_k}-u)\rangle\,\eta\,dx=0.
\]
Moreover, the weak convergence in $L^p(\Omega)$ of the sequence $\{\nabla u_{n_k}\}_{k\in\mathbb{N}}$ entails that
\[
\lim_{n\to\infty}\int_\Omega \langle |\nabla u|^{p-2}\,\nabla u,\nabla (u_{n_k}-u)\rangle\,\eta\,dx=0.
\]
Therefore, by combining the last two limits we can conclude
\[
\lim_{n\to\infty}\int_\Omega \langle |\nabla u_{n_k}|^{p-2}\,\nabla u_{n_k}-|\nabla u|^{p-2}\,\nabla u,\nabla (u_{n_k}-u)\rangle\,\eta\,dx=0,
\]
as well. By standard monotonicity inequalities for the convex power $z\mapsto |z|^p$, we deduce the claimed result.
\end{proof}
We conclude this section by recording the following interesting consequence of Proposition \ref{prop:2}
\begin{coro}
Let $1<p<\infty$ and let $\Omega\subseteq\mathbb{R}^N$ be an open set. Then
\[
\lambda\in \mathfrak{S}_{\rm eigen,p}(\Omega),
\]
if and only if there exists a regular constrained Palais-Smale sequence $\{u_n\}_{n\in\mathbb{N}}\subseteq W^{1,p}_0(\Omega)$ at the level $\lambda$.
\end{coro}
\begin{proof}
If $\lambda$ is an eigenvalue, then by definition it admits an associated eigenfunction $u\in W^{1,p}_0(\Omega)\setminus\{0\}$. Accordingly, if we define the constant sequence
\[
u_n=\frac{u}{\|u\|_{L^p(\Omega)}},\qquad \text{for every}\ n\in\mathbb{N},
\]
this turns out to be a regular constrained Palais-Smale sequence at the level $\lambda$.
\vskip.2cm\noindent
We prove the converse implication.
By definition and Proposition \ref{prop:2}, there exists a subsequence $\{u_{n_k}\}_{k\in\mathbb{N}}\subseteq\{u_n\}_{n\in\mathbb{N}}$ and a function $u\in W^{1,p}_0(\Omega)\setminus\{0\}$ such that
\begin{equation}
\label{azzenzero}
\lim_{k\to\infty} \|u_{n_k}-u\|_{L^p(\Omega\cap B_R)}=\lim_{k\to\infty} \|\nabla u_{n_k}-\nabla u\|_{L^p(\Omega\cap B_R)}=0,\qquad \text{for every}\ R>0.
\end{equation}
Moreover, by definition of Palais-Smale sequence we have in particular
\begin{equation}
\label{azzozzone}
\lim_{k\to\infty} \left[\int_\Omega \langle |\nabla u_{n_k}|^{p-2}\,\nabla u_{n_k},\nabla \varphi\rangle\,dx-\lambda\,\int_\Omega |u_{n_k}|^{p-2}\,u_{n_k}\,\varphi\,dx\right]=0,\qquad \text{for every}\ \varphi\in C^\infty_0(\Omega).
\end{equation}
We fix $\psi\in C^\infty_0(\Omega)$, since it has compact support there exists $R_0>0$ large enough so that $\psi\in C^\infty_0(\Omega\cap B_{R_0})$, as well. Thus, by \eqref{azzenzero} we can pass to the limit in \eqref{azzozzone} tested with $\psi$. We obtain
\[
\int_\Omega \langle |\nabla u|^{p-2}\,\nabla u,\nabla \psi\rangle\,dx=\lambda\,\int_\Omega |u|^{p-2}\,u\,\psi\,dx.
\] 
Since $u\not\equiv 0$ and $\psi$ is arbitrary, we finally get that $\lambda$ is an eigenvalue.
\end{proof}

\section{Proof of the Main Theorem}
\label{sec:3}

The result of the Main Theorem will be obtained by joining the next two lemmas, proving that
\[
\mathcal{E}_p(\Omega)\le \inf\mathfrak{S}_{{\rm ess},p}(\Omega)\qquad \text{and}\qquad \mathcal{E}_p(\Omega)\ge \inf\mathfrak{S}_{{\rm ess},p}(\Omega),
\]
respectively. In proving the second inequality, we will also show that $\mathcal{E}_p(\Omega)\in \mathfrak{S}_{{\rm ess},p}(\Omega)$, whenever this constant is finite.
\begin{lm}
\label{lm:1}
Let $1<p<\infty$ and let $\Omega\subseteq\mathbb{R}^N$ be an open set.
Then we have
\[
\mathcal{E}_p(\Omega)\le \inf \mathfrak{S}_{{\rm ess},p}(\Omega).
\]
\end{lm}
\begin{proof}
It is sufficient to prove that every constrained Palais-Smale sequence at the level $\lambda<\mathcal{E}_p(\Omega)$ is regular. Let $\{u_n\}_{n\in\mathbb{N}}$ be such a sequence: in particular, it is bounded in $W^{1,p}_0(\Omega)$. Thus, there exists a subsequence $\{u_{n_k}\}_{k\in\mathbb{N}}\subseteq \{u_n\}_{n\in\mathbb{N}}$ and a function $u\in W^{1,p}_0(\Omega)$, such that $\{u_{n_k}\}_{k\in\mathbb{N}}$ weakly converges in $W^{1,p}(\Omega)$ to $u$. By using Lemma \ref{lm:lemmadelca} for this subsequence, we have that $u\not\equiv 0$. We can finally apply Lemma \ref{lm:mizuno} to $\{u_{n_k}\}_{k\in\mathbb{N}}$ and conclude that $\{u_n\}_{n\in\mathbb{N}}$ is regular.
\end{proof}

\begin{lm}
\label{lm:2}
Let $1<p<\infty$ and let $\Omega\subseteq\mathbb{R}^N$ be an open set.
Then we have
\begin{equation}
\label{lm2.3.0}
\mathcal{E}_p(\Omega)\ge \inf \mathfrak{S}_{{\rm ess},p}(\Omega).
\end{equation}
Moreover, if $\mathcal{E}_p(\Omega)<+\infty$, then $\mathcal{E}_p(\Omega)\in \mathfrak{S}_{{\rm ess},p}(\Omega)$.
\end{lm}
\begin{proof}
We may assume  that 
\[
\mathcal{E}_p(\Omega)<+\infty,
\]
since otherwise the desired inequality is obvious. We will prove that in this case singular Palais-Smale sequences
at level $\mathcal{E}_p(\Omega)$ do exist, so that $\mathcal{E}_p(\Omega)\in \mathfrak{S}_{{\rm ess},p}(\Omega)$.
In order to construct such a sequence, we will use a variational method: in particular, for $p=2$ the resulting proof will be different from that of Persson.
\vskip.2cm\noindent
{\it Step 1: preparation}. We consider the weighted Sobolev space $W^{1,p}_0(\Omega;|x|)$, i.e. the closure of $C^\infty_0(\Omega)$ in
\[
W^{1,p}(\Omega;|x|)=\left\{\varphi\in W^{1,p}(\Omega)\, :\, \int_\Omega |x|\,|\varphi|^p\,dx<+\infty\right\},
\]
endowed with the norm
\[
\|\varphi\|_{W^{1,p}(\Omega;|x|)}:=\left(\int_\Omega |\nabla \varphi|^p\,dx+\int_\Omega (1+|x|)\,|\varphi|^p\,dx\right)^\frac{1}{p}.
\]
For every $R>0$ and for every $\varepsilon>0$, we take a cut-off function $\zeta_R\in C_0^\infty(B_{2R})$, with 
\[
0\le\zeta\le1,\qquad \zeta_R=1\ \text{in}\ B_R,\qquad |\nabla \zeta_R|\le \frac{2}{R},
\]
and we set
\[
V_{\varepsilon,R}(x)=\sqrt{R}\, \zeta_R(x)+\varepsilon\, |x|\,,\qquad \text{for every}\ x\in \mathbb R^N.
\]
We then consider the following minimization problem
\begin{equation}
\label{en-eps-R}
\lambda_p(\Omega;V_{\varepsilon,R})= \inf_{\varphi\in W^{1,p}_0(\Omega;|x|)}
\left\{
\int_\Omega|\nabla \varphi|^p\,dx +\int_\Omega V_{\varepsilon,R}\,|\varphi|^p\,dx\, :\, 
\int_\Omega|\varphi|^p\,dx=1
\right\}\,.
\end{equation}
We observe at first that, for every $\varepsilon>0$ and every $R>0$, there exists a minimizer $u_{\varepsilon,R}\in W^{1,p}_0(\Omega;|x|)$ for the previous problem. Indeed, we can apply the Direct Method in the Calculus of Variations, by exploiting the compact embedding
\[
W^{1,p}_0(\Omega;|x|)\hookrightarrow L^p(\Omega),
\]
see for example \cite[Proposition 2.7]{BraBriPri_low}. By minimality, we have in particular the energy identity
\begin{equation}
\label{EI}
\int_\Omega |\nabla u_{\varepsilon,R}|^p\,dx+\int_\Omega V_{\varepsilon,R}\,|u_{\varepsilon,R}|^p\,dx=\lambda_p(\Omega;V_{\varepsilon,R}),
\end{equation}
and the Euler-Lagrange equation
\begin{equation}
\label{EL}
\begin{split}
\int_\Omega \langle|\nabla u_{\varepsilon,R}|^{p-2}\,\nabla u_{\varepsilon,R},\nabla\varphi\rangle\,dx&+\int_\Omega V_{\varepsilon,R}\,|u_{\varepsilon,R}|^{p-2}\,u_{\varepsilon,R}\,\varphi\,dx\\
&=\lambda_p(\Omega;V_{\varepsilon,R})\,\int_\Omega |u_{\varepsilon,R}|^{p-2}\,u_{\varepsilon,R}\,\varphi\,dx,
\end{split}
\end{equation}
for every $\varphi\in W^{1,p}_0(\Omega;|x|)$.
\vskip.2cm\noindent
{\it Step 2: construction of the sequence}. For each $n\in\mathbb{N}\setminus\{0\}$, we now choose $v_n\in C^\infty_0(\Omega\setminus \overline{B_{2n}})$
\[
\int_{\Omega\setminus \overline{B_{2n}}} |\nabla v_n|^p\,dx \le \lambda_p\left(\Omega\setminus \overline{B_{2n}}\right) + \frac{1}{n},\qquad\int_{\Omega\setminus B_{2n}}|v_n|^{p}\,dx=1.
\]
Since $v_n$ has compact support, there exists $R_n>2\,n$ such that $v_n\in C^\infty_0(B_{R_n})$, as well.
For all $n\in\mathbb N\setminus\{0\}$, we choose 
\[
\varepsilon_n = \frac{1}{R_n^2},
\]
and we set for simplicity
\[
	\lambda_n= \lambda_p(\Omega\mathbin;V_{\varepsilon_n,n})\,,
	\qquad
	u_n=u_{\varepsilon_n, n}\qquad	\text{and}\qquad
	V_n=V_{\varepsilon_n,n}=\sqrt{n}\, \zeta_n+\frac{1}{R_n^2}\, |x|.
\]
In other words, we consider the quantities introduced in {\it Step 1} with the choices $R=n$ and $\varepsilon=\varepsilon_n$.
We claim that (a suitable subsequence of) $\{u_n\}_{n\ge 1}$ is the claimed singular constrained Palais-Smale sequence at level $\mathcal{E}_p(\Omega)$. We thus need to verify that $\{u_n\}_{n\ge 1}$ has the following properties:
\begin{equation}
\label{PS1}
\|u_n\|_{L^p(\Omega)}=1,\qquad \text{for every}\ n\in\mathbb{N}\setminus\{0\},
\tag{$PS_1$}
\end{equation}
\begin{equation}
\label{PS2}
\lim\limits_{n\to\infty} \int_\Omega |\nabla u_n|^p\,dx=\mathcal{E}_p(\Omega),
\tag{$PS_2$}
\end{equation}
\begin{equation}
\label{PS3} 
\lim_{n\to\infty} \Big\|-\Delta_p u_n-\lambda\,|u_n|^{p-2}\,u_n\Big\|_{W^{-1,p'}(\Omega)}=0,
\tag{$PS_3$}
\end{equation}
and finally
\begin{equation}
\label{PS4}
\lim_{n\to\infty}  \|u_n\|_{L^p(\Omega\cap B_R)}=0,\qquad \text{for every}\ R>0.
\tag{$PS_4$}
\end{equation}
We notice that \eqref{PS1} is true by construction. 
\par
In order to prove the other properties, we collect some additional facts. Observe that $\zeta_n$ and $v_n$ have disjoint supports, this implies that 
\[
V_n(x)\,v_n(x)= \varepsilon_n\, |x|\,v_n(x),\qquad \text{for every}\ x\in\Omega.
\]
Thus, by using $v_n$ as a competitor for \eqref{en-eps-R}
with $\varepsilon=\varepsilon_n=1/R_n^2$ and $R=n$, we obtain
\[
\begin{split}
	\lambda_n&\le \int_{\Omega\setminus\overline{B_{2n}}}
	|\nabla v_n|^p\,dx+\frac{1}{R_n^2}\,\int_{\Omega\setminus\overline{B_{2n}}}|x|\,|v_n|^p\,dx\\
	&\le \lambda_p(\Omega\setminus \overline{B_{2n}}) + \frac{1}{n}+ \frac{1}{R_n^2}\,\left(\max_{x\in B_{R_n}} |x|\right)\,\int_{B_{R_n}} |v_n|^p\,dx=\lambda_p(\Omega\setminus \overline{B_{2n}}) + \frac{1}{n}+\frac{1}{R_n}.
\end{split}	
\]
In the second inequality we used that each $v_n$ is compactly supported in $B_{R_n}$, while in
last equality we used that $\|v_n\|_{L^p(B_{R_n})}=1$.
Whence, by taking the limit as $n$ goes to $\infty$, it follows that
\[
\limsup_{n\to\infty} \lambda_n\le \lim_{n\to\infty}\left(\lambda_p(\Omega\setminus \overline{B_{2n}}) + \frac{1}{n}+\frac{1}{R_n}\right)=\mathcal{E}_p(\Omega),
\]
thanks to the definition of $\mathcal{E}_p(\Omega)$. This in particular implies that there exists a constant $C$ not depending on $n$, such that
\begin{equation}
\label{uniforme}
\int_\Omega |\nabla u_n|^p\le \lambda_n\le C,\qquad \text{for every}\ n\in\mathbb{N}\setminus\{0\}.
\end{equation}
Observe that the leftmost inequality follows from \eqref{EI} and the fact that $V_n\ge 0$.
\vskip.2cm\noindent
{\it Step 3: verification of $\eqref{PS4}$.} From the energy identity \eqref{EI} and \eqref{uniforme}, we obtain in particular
\[
\sqrt{n}\, \int_{\Omega\cap B_n} |u_n|^p\,dx\le \sqrt{n}\, \int_\Omega \zeta_n\,|u_n|^p\,dx\le \lambda_n\le C.
\]
Thus, if we fix $R>0$, we obtain for every $n\ge R$
\[
\int_{\Omega\cap B_R} |u_n|^p\,dx\le \frac{C}{\sqrt{n}},
\]
so that
\[
\lim_{n\to\infty} \|u_n\|_{L^p(\Omega\cap B_R)}=0,
\]
as desired.
\vskip.2cm\noindent
{\it Step 4: verification of \eqref{PS2}}. From {\it Step 2} and \eqref{EI}, we already know that 
\[
\limsup_{n\to\infty} \int_\Omega |\nabla u_n|^p\,dx\le \mathcal{E}_p(\Omega).
\]
We argue by contradiction and suppose that 
\[
\limsup_{n\to\infty} \int_\Omega |\nabla u_n|^p\,dx< \mathcal{E}_p(\Omega).
\]
By recalling \eqref{PS1}, up to taking a subsequence, we can apply Lemma \ref{lm:lemmadelca} and assures that $\{u_n\}_{n\in\mathbb{N}}$ has a nontrivial weak limit $u\in W^{1,p}_0(\Omega)$, such that
\[
\|u\|_{L^p(\Omega)}\ge \mathfrak{d}>0.
\]
By using that 
\[
\lim_{R\to +\infty} \|u\|_{L^p(\Omega\cap B_R)}=\|u\|_{L^p(\Omega)},
\]
we can in particular find a radius $\overline{R}>0$ such that
\[
\|u\|_{L^p(\Omega\cap B_{\overline{R}})}\ge \frac{\mathfrak{d}}{2}.
\]
On the other hand, by using \eqref{PS4} and the lower semicontinuity of the $L^p$ norm with respect to the weak convergence, we get
\[
0=\lim_{n\to\infty} \|u_n\|_{L^p(\Omega\cap B_{\overline{R}})}\ge \|u\|_{L^p(\Omega\cap B_{\overline{R}})}\ge \frac{\mathfrak{d}}{2}>0.
\]
This gives a contradiction. We thus must have 
\[
\limsup_{n\to\infty} \int_\Omega |\nabla u_n|^p\,dx= \mathcal{E}_p(\Omega).
\]
Possibly taking a subsequence, we then obtain the desired property \eqref{PS2}.
\vskip.2cm\noindent
{\it Step 5: verification of \eqref{PS3}}. This is the most delicate step.  We first observe that
\begin{equation}
\label{limen2}
\lim_{n\to\infty} \lambda_n = \lim_{n\to\infty}\left[\int_\Omega
|\nabla u_n|^p\,dx+\int_\Omega V_n\,|u_n|^p\,dx\right]=\mathcal{E}_p(\Omega),
\end{equation}
and 
\begin{equation}
\label{panino}
\lim_{n\to\infty}\int_\Omega V_n\,|u_n|^p\,dx=0,
\end{equation}
as a consequence of {\it Step 2} and {\it Step 4}. Actually, by appealing to Proposition \ref{prop:salvaculo}, 
we can obtain the following upgrade of \eqref{panino}
\begin{equation}
\label{ind3}
\lim_{n\to\infty} \int_\Omega V_n^k\,|u_n|^p\,dx=0\,, \qquad\text{for every}\ k\in\mathbb{N}\setminus\{0\}.
\end{equation}
Observe that the latter obviously implies that
\begin{equation}
\label{quasisoluz}
\lim_{n\to\infty}\int_\Omega V_n^{p'}\,|u_n|^p\,dx=0,
\end{equation}
as well.
Indeed, there exists $k\in\mathbb N\setminus\{0\}$ such that $k<p'\le k+1$. By using the elementary inequality
\[
a^{p'}\le a^k+a^{k+1},\qquad \text{for every}\ a\ge 0,
\]
we get
\[	
\int_\Omega V_n^{p'}\,|u_n|^p\,dx\le 
\int_\Omega V_n^{k}\,|u_n|^p\,dx+\int_\Omega V_n^{k+1}\,|u_n|^p\,dx.
\]	
Whence \eqref{quasisoluz} follows by \eqref{ind3}.
\par
It is now time to conclude the proof. We take $\varphi\in C^\infty_0(\Omega)$, by using the Euler-Lagrange equation \eqref{EL}, we have
\[
\begin{split}
\left\langle-\Delta_p u_n-\mathcal{E}_p(\Omega)\,|u_n|^{p-2}\,u_n,\varphi\right\rangle_{(W^{-1,p'}(\Omega),W^{1,p}_0(\Omega))}&=\int_\Omega \langle |\nabla u_n|^{p-2}\,\nabla u_n,\nabla \varphi\rangle\,dx\\
&-\mathcal{E}_p(\Omega)\,\int_\Omega |u_n|^{p-2}\,u_n\,\varphi\,dx\\
&=-\int_\Omega V_n\,|u_n|^{p-2}\,u_n\,\varphi\,dx\\
&+\Big(\lambda_n-\mathcal{E}_p(\Omega)\Big)\,\int_\Omega |u_n|^{p-2}\,u_n\,\varphi\,dx.
\end{split}
\]
By using H\"older's inequality, we get
\[
\begin{split}
\left|\left\langle-\Delta_p u_n-\mathcal{E}_p(\Omega)\,|u_n|^{p-2}\,u_n,\varphi\right\rangle_{(W^{-1,p'}(\Omega),W^{1,p}_0(\Omega))}\right|&\le \left(\int_\Omega V_n^{p'}\,|u_n|^p\,dx\right)^\frac{p-1}{p}\,\|\varphi\|_{L^p(\Omega)}\\
&+|\lambda_n-\mathcal{E}_p(\Omega)|\,\|\varphi\|_{L^p(\Omega)}.
\end{split}
\]
By using that $\|\varphi\|_{L^p(\Omega)}\le \|\varphi\|_{W^{1,p}(\Omega)}$, the arbitrariness of $\varphi\in C^\infty_0(\Omega)$ and both \eqref{limen2} and \eqref{quasisoluz}, we obtain
\[
\lim_{n\to\infty}\left\|-\Delta_p u_n-\mathcal{E}_p(\Omega)\,|u_n|^{p-2}\,u_n\right\|_{W^{-1,p'}(\Omega)}=0.
\]
This concludes the proof.
\end{proof}
\begin{oss}
We observe that the singular sequence at the level $\mathcal{E}_p(\Omega)$ constructed in the previous proof actually enjoys the following stronger property
\[
\lim_{n\to\infty}\left\|-\Delta_p u_n-\mathcal{E}_p(\Omega)\,|u_n|^{p-2}\,u_n\right\|_{L^{p'}(\Omega)}=0.
\]
\end{oss}

\section{Consistency for the case $p=2$}
\label{sec:4}
In this section, we show that our definition of essential spectrum is consistent with the usual one for $p=2$. 
At a first glance, by comparing the definition of singular Weyl sequence  and Definition \ref{defi:PS}, we can notice a certain analogy between the two. 
However, in principle the last concept seems slightly weaker, because it does not require $-\Delta u_n-\lambda\,u_n\in L^2(\Omega)$. 
\par
Indeed, in general {\it it is not true} that a singular constrained Palais-Smale sequences is also a singular Weyl sequence. Nevertheless, every such a sequence can be ``slightly amended'' in order to fulfill the additional requirement that $-\Delta u_n-\lambda\,u_n$ goes to $0$ in $L^2(\Omega)$, as well. This is the content of the following
\begin{lm}[``Weyl correction'']
\label{lm:correction}
Let $\Omega\subseteq\mathbb{R}^N$ be an open set. Let us suppose that there exists a singular constrained Palais-Smale sequence $\{u_n\}_{n\in\mathbb{N}}\subseteq W^{1,2}_0(\Omega)$ at the level $\lambda$. Then there also exists a singular Weyl sequence at the same level.
\end{lm}
\begin{proof}
We set 
\[
h_n:=-\Delta u_n-\lambda\,u_n\in W^{-1,2}(\Omega).
\]
By assumption, we know that $h_n$ converges to $0$ in $W^{-1,2}(\Omega)$. We now take $U_n\in W^{1,2}_0(\Omega)$ the unique minimizer of the strictly convex functional
\[
\frac{1}{2}\,\int_\Omega |\nabla\varphi|^2\,dx+\frac{1}{2}\,\int_\Omega |\varphi|^2\,dx-\langle h_n,\varphi\rangle_{(W^{-1,2}(\Omega),W^{1,2}_0(\Omega))},\qquad \text{for every}\ \varphi\in W^{1,2}_0(\Omega).
\]
In particular, $U_n$ weakly solves
\[
-\Delta U_n+U_n=h_n,\qquad \text{in}\ \Omega,
\]
and it thus verifies
\[
\|U_n\|_{W^{1,2}(\Omega)}^2=\int_\Omega |\nabla U_n|^2\,dx+\int_\Omega |U_n|^2\,dx\le \|h_n\|_{W^{-1,2}(\Omega)}\,\|U_n\|_{W^{1,2}(\Omega)},
\]
which implies
\begin{equation}
\label{correttore}
\|U_n\|_{W^{1,2}(\Omega)}\le \|h_n\|_{W^{-1,2}(\Omega)}.
\end{equation}
Observe in particular that 
\[
1-\|U_n\|_{L^2(\Omega)}\le\|u_n-U_n\|_{L^2(\Omega)}\le 1+\|U_n\|_{L^2(\Omega)},
\]
thus from \eqref{correttore} we get that 
\begin{equation}
\label{norma1}
\lim_{n\to\infty} \|u_n-U_n\|_{L^2(\Omega)}=1.
\end{equation}
We are ready to construct the singular Weyl sequence: we set 
\[
v_n=\frac{u_n-U_n}{\|u_n-U_n\|_{L^2(\Omega)}}\in W^{1,2}_0(\Omega),
\]
and observe that each $v_n$ has unit $L^2$ norm. In addition, $\{v_n\}_{n\in\mathbb{N}}$ converges weakly to $0$ in $L^2(\Omega)$, thanks to the properties of both $u_n$ and $U_n$, together with \eqref{norma1}. Moreover, by construction we have
\[
-\Delta v_n-\lambda\,v_n=\frac{-\Delta u_n-\lambda\,u_n+\Delta U_n+\lambda\,U_n}{\|u_n-U_n\|_{L^2(\Omega)}}=\frac{(1+\lambda)\,U_n}{\|u_n-U_n\|_{L^2(\Omega)}}\in L^2(\Omega).
\]
In particular, we have $-\Delta v_n\in L^2(\Omega)$. By recalling \eqref{norma1} and \eqref{correttore}, we also get that $-\Delta v_n-\lambda\,v_n$ converges to $0$ in $L^2(\Omega)$, as $n$ goes to $\infty$. This concludes the proof.
\end{proof}
We can finally get the following interesting result, showing that our definition of essential spectrum coincides with the usual one in the linear case.
\begin{teo}
\label{teo:essenza2}
Let $\Omega\subseteq\mathbb{R}^N$ be an open set. We denote by $\mathfrak{S}_{\rm ess}(\Omega)$ the essential spectrum of the Dirichlet-Laplacian on $\Omega$. Then we have
\[
\mathfrak{S}_{{\rm ess},2}(\Omega)=\mathfrak{S}_{{\rm ess}}(\Omega).
\]
\end{teo}
\begin{proof}
As recalled in the Introduction, by the {\it Weyl criterion} we have that $\lambda\in \mathfrak{S}_{\rm ess}(\Omega)$ if and only if there exists a singular Weyl sequence at the level $\lambda$. Thus, in light of our definition of $\mathfrak{S}_{{\rm ess},2}(\Omega)$ and Lemma \ref{lm:correction}, we get
\[
\mathfrak{S}_{{\rm ess},2}(\Omega)\subseteq\mathfrak{S}_{{\rm ess}}(\Omega).
\]
The converse inclusion is standard, we give the details for completeness. Let us take $\lambda\in \mathfrak{S}_{{\rm ess}}(\Omega)$ and consider a singular Weyl sequence $\{u_n\}_{n\in\mathbb{N}}\subseteq L^2(\Omega)$. Observe that 
\[
\|u_n\|_{L^2(\Omega)}=1,\qquad \text{for every}\ n\in\mathbb{N}.
\]
By (W4), we also get
\[
\lim_{n\to\infty} \|-\Delta u_n-\lambda\,u_n\|_{W^{-1,2}(\Omega)}\le \lim_{n\to\infty} \|-\Delta u_n-\lambda\,u_n\|_{L^2(\Omega)}=0.
\]
In the last inequality, we used that $L^2(\Omega)\subseteq W^{-1,2}(\Omega)$, with 
\[
\|\varphi\|_{W^{-1,2}(\Omega)}\le \|\varphi\|_{L^2(\Omega)},\qquad \text{for every}\ \varphi\in L^2(\Omega).
\]
Indeed, for every $\varphi\in L^2(\Omega)$ and every $\psi\in W^{1,2}_0(\Omega)$, we have 
\[
\left|\int_\Omega \varphi\,\psi\,dx\right|\le \|\varphi\|_{L^2(\Omega)}\,\|\psi\|_{L^2(\Omega)}\le \|\varphi\|_{L^2(\Omega)}\,\|\psi\|_{W^{1,2}(\Omega)}.
\]
In order to prove that $\{u_n\}_{n\in\mathbb{N}}$ is a singular constrained Palais-Smale sequence at the level $\lambda$, we still need to prove that 
\begin{equation}
\label{renato}
\lim_{n\to\infty} \int_\Omega |\nabla u_n|^2\,dx=\lambda\qquad \text{and}\qquad \lim_{n\to\infty}  \|u_n\|_{L^2(\Omega\cap B_R)}=0,\qquad \text{for every}\ R>0.
\end{equation}
The first fact follows by using the strong $L^2(\Omega)$ convergence of $-\Delta u_n-\lambda\,u_n$. Indeed, the latter implies that 
\[
\begin{split}
\int_\Omega |\nabla u_n|^2\,dx-\lambda=\left\langle-\Delta u_n-\lambda\,u_n,u_n\right\rangle_{(L^2(\Omega),L^2(\Omega))}&=o(1)\,\|u_n\|_{L^2(\Omega)}=o(1).
\end{split}
\]
This identity gives the desired property of the Dirichlet integral. 
\par
The second property in \eqref{renato} can be proved by appealing to Lemma \ref{lm:mizuno}. Indeed, observe that the weak convergence to $0$ in $L^2(\Omega)$ of $\{u_n\}_{n\in\mathbb{N}}$ implies that
\[
\lim_{n\to\infty} \int_\Omega \langle \nabla u_n,\phi\rangle\,dx=-\lim_{n\to\infty} \int_\Omega u_n\,\mathrm{div}\,\phi\,dx=0,\qquad \text{for every}\ \phi\in C^\infty_0(\Omega;\mathbb{R}^N).
\]
Thus, we have that $\{u_n\}_{n\in\mathbb{N}}$ weakly converges to $0$ in $W^{1,2}(\Omega)$. By Lemma \ref{lm:mizuno}, we obtain the desired local strong convergence.
In conclusion, we get
\[
\mathfrak{S}_{{\rm ess},2}(\Omega)\supseteq\mathfrak{S}_{{\rm ess}}(\Omega),
\]
as well.
\end{proof}

\section{A case study: the spectrum of a rectilinear strip}
\label{sec:5}
In this section we will compute the full spectrum of the Dirichlet $p-$Laplacian for a rectilinear strip. Throughout the whole secion, we will use the notation
\[
\mathcal{S}_\alpha=(-\alpha,\alpha)\times\mathbb{R},
\]
where $\alpha>0$. We also fix the following notation for a rectangle: for $T>0$, we set
\[
\mathcal{S}_{\alpha,T}=(-\alpha,\alpha)\times(-T,T).
\]
Moreover, we introduce the one-dimensional Poincar\'e-Sobolev constant 
\[
\pi_p:=\inf_{\varphi\in C^\infty_0(\mathrm{I})} \Big\{\|\varphi'\|_{L^p(\mathrm{I})}\, :\, \|\varphi\|_{L^p(\mathrm{I})}=1\Big\},\quad  \text{where }\mathrm{I}=\left(-\frac{1}{2},\frac{1}{2}\right).
\]
With this definition, it is not difficult to show that
\begin{equation}
\label{brunopini}
\lambda_p(\mathcal{S}_\alpha) = \lambda_p((-\alpha,\alpha))=\left(\frac{\pi_p}{2\,\alpha}\right)^p,
\end{equation}
see for example \cite[Lemma A.2]{Bra_bus}.
\vskip.2cm\noindent
We first need to recall some regularity results for eigenfunctions on rectangles. Though not optimal, the following statement will be largely sufficient for our scopes.
\begin{lm}
\label{lm:C1}
Let $1<p<\infty$ and let $u\in W^{1,p}_0(\mathcal{S}_{\alpha,T})$ be a solution of 
\[
-\Delta_p u=\lambda\,|u|^{p-2}\,u,\qquad \text{in}\ \mathcal{S}_{\alpha,T},
\]
for some $\lambda>0$. Then $u\in C^1(\overline{\mathcal{S}_{\alpha,T}})$. The same result is true also if $u\in W^{1,p}_0(\mathcal{S}_{\alpha})$ and it weakly solves the above equation in $\mathcal{S}_\alpha$.
\end{lm}
\begin{proof}
In the ``bounded'' case of $\mathcal{S}_{\alpha,T}$, this can be deduced from classical interior $C^{1,\alpha}$ estimates like those of \cite[Theorem 2]{DiB}, by using a simple ``extension by odd reflection'' argument. See for instance \cite[Lemma 2.1]{Bob}.
\par
For the case $\mathcal{S}_\alpha$, we can proceed similarly: we extend to the larger strip $\mathcal{S}_{3\alpha}=(-3\,\alpha,3\,\alpha)\times\mathbb{R}$ the function $u$, by using two odd reflections in the $x_1$ variable. If we indicate by $U$ the resulting function, this is still a solution of the same equation in $\mathcal{S}_{3\,\alpha}$. In particular, by \cite[Theorem 2]{DiB} we get
\[
U\in C^1(\overline{\mathcal{S}_{\ell,T}}),\qquad \text{for every}\ 0<\ell<3\,\alpha,\ T>0.
\]
By choosing $\ell=\alpha$ and using the arbitrariness of $T>0$, we conclude.
\end{proof}
In what follows, for every $\varepsilon>0$ we set
\[
H_\varepsilon(z)=\frac{1}{p}\,(\varepsilon^2+|z|^2)^\frac{p}{2},\qquad \text{for every}\ z\in\mathbb{R}^N.
\]
The following approximation result will useful in a while.
\begin{lm}
\label{lm:approx2}
Let $1<p<\infty$, if $u\in W^{1,p}_0(\mathcal{S}_\alpha)$ is a weak solution of 
\[
-\Delta_p u=\lambda\,|u|^{p-2}\,u,\qquad \text{in}\ \mathcal{S}_\alpha,
\]
for some $\lambda>0$. Let $T>0$, for every $\varepsilon>0$ we consider the functional
\[
\mathcal{F}_\varepsilon(\varphi):=\int_{\mathcal{S}_{\alpha,T}} H_{\varepsilon}(\nabla\varphi)\,dx-\lambda\,\int_{\mathcal{S}_{\alpha,T}} |u|^{p-2}\,u\,\varphi\,dx,\qquad \text{for every}\ \varphi\in W^{1,p}(\mathcal{S}_{\alpha,T}).
\]
Then the problem
\[
\min_{\varphi\in W^{1,p}(\mathcal{S}_{\alpha,T})} \left\{\mathcal{F}_\varepsilon(\varphi)\, :\, \varphi-u\in W^{1,p}_0(\mathcal{S}_{\alpha,T})\right\},
\]
admits a unique solution $u_\varepsilon$, which weakly solves
\[
-\mathrm{div}\nabla H_\varepsilon(\nabla u_\varepsilon)=\lambda\,|u|^{p-2}\,u,\qquad \mbox{in}\ \mathcal{S}_{\alpha,T}.
\]
Moreover, we have $u_\varepsilon \in C^2(\overline{\mathcal{S}_{\alpha,t}})$, for every $0<t<T$, and 
\[
\frac{\partial u_\varepsilon}{\partial x_2}=0,\qquad \text{on}\ \Big(\{-\alpha\}\times[-t,t]\Big)\cup\Big(\{\alpha\}\times[-t,t]\Big).
\]
Finally, there exists an infinitesimal sequence of positive numbers $\{\varepsilon_k\}_{k\in\mathbb{N}}$ such that
\[
\lim_{\varepsilon\to 0} \|u_{\varepsilon_k}-u\|_{W^{1,p}(\mathcal{S}_{\alpha,T})}=0.
\]
\end{lm}
\begin{proof}
Existence of a solution $u_\varepsilon$ can be inferred by using the Direct Method in the Calculus of Variations (see for example \cite[Theorem 4.3.3]{BraBook}). 
The uniqueness of the solution is a consequence of the strict convexity of the functional. Finally, the minimality condition is precisely given by the Euler-Lagrange equation
\[
-\mathrm{div}\nabla H_\varepsilon(\nabla u_\varepsilon)=\lambda\,|u|^{p-2}\,u,\qquad \text{in}\ \mathcal{S}_{\alpha,T},
\]
in weak form. 
\vskip.2cm\noindent
We now want to prove that $u_\varepsilon$ has the claimed regularity. We first recall that $u\in C^1(\overline{\mathcal{S}_{\alpha,T}})$ by Lemma \ref{lm:C1}.
In particular, we have that
\[
\lambda\,|u|^{p-2}\,u\in C^{0,\beta}(\overline{\mathcal{S}_{\alpha,T}}),
\]
possibly for an exponent $0<\beta\le 1$. 
We then extend both $u_\varepsilon$ and $u$ to the larger rectangle
\[
\mathcal{S}_{3\alpha,T}=(-3\alpha,3\alpha)\times(-T,T),
\]
by odd reflections with respect to the $x_1$ variable. We call $U_\varepsilon$ and $U$ the resulting functions, then we have that $U_\varepsilon$ weakly solves
\[
-\mathrm{div}\nabla H_\varepsilon(\nabla U_\varepsilon)=\lambda\,|U|^{p-2}\,U,\qquad \text{in}\ \mathcal{S}_{3\alpha,T}.
\]
Moreover, we have $U_\varepsilon-U\in W^{1,p}_0(\mathcal{S}_{3\alpha,T})$. Observe that the right-hand side is H\"older continuous and the equation is uniformly elliptic, thus we can infer that $U_\varepsilon\in C^{2,\gamma}(\overline{\mathcal{O}})$ for some $0<\gamma<1$, for every $\mathcal{O}\Subset \mathcal{S}_{3\alpha,T}$. In particular, by choosing $0<t<T$ and taking
\[
\mathcal{O}=\mathcal{S}_{\alpha,t},
\]
we get the desired conclusion, if we observe that $U_\varepsilon=u_\varepsilon$ on $\mathcal{S}_{\alpha,t}$.
Observe that $u_\varepsilon=0$ in classical pointwise sense on the lateral boundary of $\mathcal{S}_{\alpha,t}$, with $0<t<T$. Thus, the derivative in $x_2$ identlically vanishes, for $x_1=\pm \alpha$ and $|x_2|\le t$.
\vskip.2cm\noindent
In order to infer the strong convergence, we preliminary observe that 
\begin{equation}
\label{Hconverge}
\lim_{\varepsilon\to0} \int_{\mathcal{S}_{\alpha,T}} H_{\varepsilon}(\nabla \varphi)\,dx=\frac{1}{p}\,\int_{\mathcal{S}_{\alpha,T}} |\nabla \varphi|^p\,dx,\qquad \text{for every}\ \varphi\in W^{1,p}(\mathcal{S}_{\alpha,T}).
\end{equation}
Indeed, for every $\varepsilon>0$ and $z\in\mathbb{R}^N$, we have
\[
\begin{split}
\left|H_\varepsilon(z)-\frac{|z|^p}{p}\right|=\left|\frac{1}{p}\,(\varepsilon^2+|z|^2)^\frac{p}{2}-\frac{|z|^p}{p}\right|&=\left|\int_0^\varepsilon (\tau^2+|z|^2)^\frac{p-2}{2}\,\tau\,d\tau\right|\\
&\le \int_0^\varepsilon (\tau^2+|z|^2)^\frac{p-2}{2}\,|\tau|\,d\tau\\
&\le \int_0^\varepsilon (\tau^2+|z|^2)^\frac{p-2}{2}\,(\tau^2+|z|^2)^\frac{1}{2}\,d\tau\\
&\le (\varepsilon^2+|z|^2)^\frac{p-1}{2}\,\varepsilon.
\end{split}
\] 
This permits to infer that
\[
\begin{split}
\left|\int_{\mathcal{S}_{\alpha,T}} H_{\varepsilon}(\nabla \varphi)\,dx-\frac{1}{p}\,\int_{\mathcal{S}_{\alpha,T}} |\nabla \varphi|^p\,dx\right|&\le \varepsilon\,\int_{\mathcal{S}_{\alpha,T}}(\varepsilon^2+|\nabla\varphi|^2)^\frac{p-1}{2}\,dx\\
&\le \max\left\{1,2^\frac{p-3}{2}\right\}\,\varepsilon\,\left(\varepsilon^{p-1}\,|\mathcal{S}_{\alpha,T}|+\int_{\mathcal{S}_{\alpha,T}}|\nabla\varphi|^{p-1}\,dx\right),
\end{split}
\]
from which \eqref{Hconverge} easily follows.
\par
We now prove that the family $\{u_\varepsilon\}_{0<\varepsilon\le 1}$ is bounded in $W^{1,p}(\mathcal{S}_{\alpha,T})$. By virtue of the fact that $u\in W^{1,p}_0(\mathcal{S}_{\alpha})$ and using \eqref{brunopini}, it is not difficult to see that we still have 
\begin{equation}
\label{skorpio}
\left(\frac{\pi_p}{2\alpha}\right)^{p}\,\int_{\mathcal{S}_{\alpha,T}} |\varphi|^p\,dx\le \int_{\mathcal{S}_{\alpha,T}}|\nabla \varphi|^p\,dx,\quad \text{for}\ \varphi\in W^{1,p}(\mathcal{S}_{\alpha,T})\ \text{such that}\ \varphi-u\in W^{1,p}_0(\mathcal{S}_{\alpha,T}).
\end{equation}
Thus, for every admissible $\varphi$ we have by H\"older inequality and \eqref{skorpio}
\[
\begin{split}
\mathcal{F}_\varepsilon(\varphi)&\ge \frac{1}{p}\,\int_{\mathcal{S}_{\alpha,T}} |\nabla \varphi|^p\,dx-\lambda\,\|u\|_{L^p(\mathcal{S}_{\alpha,T})}^{p-1}\,\|\varphi\|_{L^p(\mathcal{S}_{\alpha,T})}\\
&\ge \frac{1}{p}\,\int_{\mathcal{S}_{\alpha,T}} |\nabla \varphi|^p\,dx-\frac{2\,\alpha\,\lambda}{\pi_p}\,\|u\|_{L^p(\mathcal{S}_{\alpha,T})}^{p-1}\,\|\nabla\varphi\|_{L^p(\mathcal{S}_{\alpha,T})}.
\end{split}
\]
By applying Young inequality in a standard way on the last term and using that $\|u\|_{L^p(\mathcal{S}_{\alpha,T})}\le \|u\|_{L^p(\Omega)}$, we get that 
\begin{equation}
\label{coercivo}
\mathcal{F}_\varepsilon(\varphi)\ge \frac{1}{2\,p}\,\int_{\mathcal{S}_{\alpha,T}} |\nabla \varphi|^p\,dx-C_{p,\lambda,\alpha}\,\|u\|_{L^p(\Omega)}^p.
\end{equation}
We can now test the minimality of the function $u_\varepsilon$ against $u$ and then use \eqref{coercivo}. We get in particular
\[
\frac{1}{2\,p}\,\int_{\mathcal{S}_{\alpha,T}} |\nabla u_\varepsilon|^p\,dx\le C_{p,\lambda,\alpha}\,\|u\|_{L^p(\Omega)}^p+\mathcal{F}_\varepsilon(u).
\] 
By noticing that for every $0<\varepsilon\le 1$ we have
\[
\mathcal{F}_\varepsilon(u)\le \max\left\{1,2^\frac{p-2}{2}\right\}\,\left(\frac{1}{p}\,\int_{\mathcal{S}_{\alpha,T}} |\nabla u|^p\,dx+\frac{1}{p}\,|\mathcal{S}_{\alpha,T}|\right)-\lambda\,\int_{\mathcal{S}_{\alpha,T}} |u|^p\,dx,
\]
we get the claimed uniform bound on the gradients. This in turn gives the uniform bound on the $W^{1,p}$ norm, by \eqref{skorpio}.
\par
Thus, there exists an infinitesimal sequence $\{\varepsilon_k\}_{k\in\mathbb{N}}$ of positive numbers and a function $v\in W^{1,p}(\mathcal{S}_{\alpha,T})$ such that $\{u_{\varepsilon_k}\}_{k\in\mathbb{N}}$ converges to $v$, weakly in $W^{1,p}(\mathcal{S}_{\alpha,T})$ and strongly in $L^p(\mathcal{S}_{\alpha,T})$. Moreover, we still have $v-u\in W^{1,p}_0(\mathcal{S}_{\alpha,T})$.
We claim that $v=u$: for every $\varphi\in W^{1,p}(\mathcal{S}_{\alpha,T})$ such that $\varphi-u\in W^{1,p}_0(\mathcal{S}_{\alpha,T})$, we have
\begin{equation}
\label{dorme}
\begin{split}
\frac{1}{p}\,\int_{\mathcal{S}_{\alpha,T}} |\nabla v|^p\,dx&-\lambda\,\int_{\mathcal{S}_{\alpha,T}} |u|^{p-2}\,u\,v\,dx\\
&\le \liminf_{k\to\infty}\left[ \frac{1}{p}\,\int_{\mathcal{S}_{\alpha,T}} |\nabla u_{\varepsilon_k}|^p\,dx-\lambda\,\int_{\mathcal{S}_{\alpha,T}} |u|^{p-2}\,u\,u_{\varepsilon_k}\,dx\right]\\
&\le \liminf_{k\to\infty}\mathcal{F}_{\varepsilon_k}(u_{\varepsilon_k})\\
&\le \lim_{k\to\infty}\mathcal{F}_{\varepsilon_k}(\varphi)=\frac{1}{p}\,\int_{\mathcal{S}_{\alpha,T}} |\nabla \varphi|^p\,dx-\lambda\,\int_{\mathcal{S}_{\alpha,T}} |u|^{p-2}\,u\,\varphi\,dx.
\end{split}
\end{equation}
In the third inequality, we used the minimality of $u_{\varepsilon_k}$, while in the last limit we used \eqref{Hconverge}. 
From \eqref{dorme} we get that the limit function $v$ is a minimizer of the functional
\[
\mathcal{F}_0(\varphi)=\frac{1}{p}\,\int_{\mathcal{S}_{\alpha,T}} |\nabla \varphi|^p\,dx-\lambda\,\int_{\mathcal{S}_{\alpha,T}} |u|^{p-2}\,u\,\varphi\,dx,
\]
in the class of functions
\[
\Big\{\varphi\in W^{1,p}(\mathcal{S}_{\alpha,T})\, :\, \varphi-u\in W^{1,p}_0(\mathcal{S}_{\alpha,T})\Big\}.
\]
The functional $\mathcal{F}_0$ is strictly convex and the set of admissible functions is convex, thus the minimizer is unique. It can be characterized as the unique weak solution of the relevant Euler-Lagrange equation, i.e.
\[
-\Delta_p v=\lambda\,|u|^{p-2}\,u,\qquad \text{in}\ \mathcal{S}_{\alpha,T}.
\]
By assumption, we know that $u$ is a solution of this equation and thus it coincides with the unique minimizer of $\mathcal{F}_0$. This discussion entails that $v=u$ and thus $\{u_{\varepsilon_k}\}_{k\in\mathbb{N}}$ converges to $u$, weakly in $W^{1,p}(\mathcal{S}_{\alpha,T})$ and strongly in $L^p(\mathcal{S}_{\alpha,T})$.
\par
Finally, by using this fact, we get that in \eqref{dorme} every inequality is actually an equality, when $\varphi=u$. In particular, we get
\[
\lim_{k\to\infty}\left[ \frac{1}{p}\,\int_{\mathcal{S}_{\alpha,T}} |\nabla u_{\varepsilon_k}|^p\,dx-\lambda\,\int_{\mathcal{S}_{\alpha,T}} |u|^{p-2}\,u\,u_{\varepsilon_k}\,dx\right]=\frac{1}{p}\,\int_{\mathcal{S}_{\alpha,T}} |\nabla u|^p\,dx-\lambda\,\int_{\mathcal{S}_{\alpha,T}} |u|^p\,dx.
\]
Since by weak convergence in $L^p(\Omega)$ we have
\[
\lim_{k\to\infty} \lambda\,\int_{\mathcal{S}_{\alpha,T}} |u|^{p-2}\,u\,u_{\varepsilon_k}\,dx=\lambda\,\int_{\mathcal{S}_{\alpha,T}} |u|^p\,dx,
\]
we must also have 
\[
\lim_{k\to\infty} \int_{\mathcal{S}_{\alpha,T}} |\nabla u_{\varepsilon_k}|^p\,dx=\int_{\mathcal{S}_{\alpha,T}} |\nabla u|^p\,dx.
\]
By uniform convexity of $L^p(\mathcal{S}_{\alpha,T})$, the first fact eventually gives the strong convergence of the gradients.
\end{proof}
\begin{oss}
By using standard reasonings based on the uniqueness of the minimizer for the limit functional $\mathcal{F}_0$, actually it is possible to prove convergence of the whole family $\{u_\varepsilon\}_{0<\varepsilon\le 1}$. We omit this fact, since it will not be needed.
\end{oss}
We now present the key tool in order to exclude the presence of eigenvalues on the strip. This is a Pohozaev-Rellich--type identity, inspired by the papers \cite{Re} and \cite{EL}.
\begin{prop}[Pohozaev-Rellich--type identity]
\label{prop:pocozaev}
Let $1<p<\infty$, if $u\in W^{1,p}_0(\mathcal{S}_\alpha)$ is a weak solution of
\[
-\Delta_p u=\lambda\,|u|^{p-2}\,u,\qquad \text{in}\ \mathcal{S}_\alpha,
\]
we have
\begin{equation}
\label{pozzaev}
\int_{\mathcal{S}_\alpha}|\nabla u|^{p-2}\,\left|\frac{\partial u}{\partial x_2}\right|^2\, dx=0.
\end{equation}
\end{prop}
\begin{proof}
The idea of the proof is quite easy: we would like to test the weak formulation of the equation for $u$ with
\[
\varphi=x_2\,\frac{\partial u}{\partial x_2}.
\]
However, for $p\not=2$ we can not assure that this function is regular enough. For this reason, we will need to use an approximation argument, as in the proof of \cite[Theorem 1.1]{GV}. The unboundedness of the set will cause some additional troubles: here, Lemma \ref{lm:approx2} above will be crucial.
For ease of readability, we divide the proof in various steps.
\vskip.2cm\noindent
{\it Step 1: approximated identity}. For a fixed $R>0$, we consider the sequence $\{u_{\varepsilon_k}\}_{k\in\mathbb{N}}\subseteq W^{1,p}_0(\mathcal{S}_{\alpha,3R})$ obtained from Lemma \ref{lm:approx2}. We recall that we have 
\begin{equation}
\label{mercy}
\lim_{k\to\infty}\|u_{\varepsilon_k}- u\|_{W^{1,p}_0(\mathcal{S}_{\alpha,3R})} = 0.
\end{equation}
We can also suppose that 
\begin{equation}
\label{mercy2}
\lim_{k\to\infty}\nabla u_{\varepsilon_k}= \nabla u,\qquad \text{for a.\,e. in}\ x\in\mathcal{S}_{\alpha,3R},
\end{equation}
up to extract a subsequence.
We consider the ``longitudinal'' cut-off
\[
\eta_R(x_1,x_2) = \min \left\{ \frac{(2\,R-|x_2|)_+}{R},1 \right\}. 
\]
Recall that each $u_{\varepsilon_k}\in W^{1,p}_0(\mathcal{S}_{\alpha,3R})$ satisfies 
\[
\int_{\mathcal{S}_{\alpha,3R}} \langle\nabla H_{\varepsilon_k}(\nabla u_{\varepsilon_k}), \nabla \varphi\rangle\, dx = \lambda \int_{\mathcal{S}_{\alpha,3R}}|u|^{p-2} u \,\varphi\, dx,\quad  \text{ for every }\varphi \in W^{1,p}_0(\mathcal{S}_{\alpha,3R}). 
\]
We plug-in the following test function
\[
\varphi_\varepsilon  = \frac{\partial u_{\varepsilon_k}}{\partial x_2 }\, x_2\, \eta_{R},
\]
extended by zero for $2\,R<|x_2|<3\,R$.
Observe that this is feasible, thanks to Lemma \ref{lm:approx2} and the regularity of the boundary of $\mathcal{S}_{\alpha,3R}$. 
Accordingly, we get
\begin{equation}\label{testo}
\int_{\mathcal{S}_{\alpha,3R}} \left\langle \nabla H_{\varepsilon_k}(\nabla u_{\varepsilon_k}), \nabla\left( \frac{\partial u_{\varepsilon_k}}{\partial x_2}\,x_2\, \eta_{R}\right) \right\rangle\, dx = \lambda \int_{\mathcal{S}_{\alpha,3R}}|u|^{p-2} u \,\frac{\partial u_{\varepsilon_k}}{\partial x_2 }\, x_2\, \eta_{R}\, dx.
\end{equation}
We observe that
\[
\left\langle \nabla H_{\varepsilon_k}(\nabla u_{\varepsilon_k}), \nabla \frac{\partial u_{\varepsilon_k}}{\partial x_2} \right\rangle = \frac{\partial}{\partial x_2} H_{\varepsilon_k}(\nabla u_{\varepsilon_k}).
\]
Hence, we can rewrite the left-hand side in \eqref{testo} as 
\[
\begin{split}
\int_{\mathcal{S}_{\alpha,3R}} \left\langle \nabla H_{\varepsilon_k}(\nabla u_{\varepsilon_k}), \nabla\left( \frac{\partial u_{\varepsilon_k}}{\partial x_2}\,x_2\, \eta_{R}\right) \right\rangle\, dx  &=\int_{\mathcal{S}_{\alpha,3R}} \frac{\partial}{\partial x_2} H_{\varepsilon_k}(\nabla u_{\varepsilon_k}) \, x_2\, \eta_{R}\,  dx \\
&+ \int_{\mathcal{S}_{\alpha,3R}}\left(|\nabla u_{\varepsilon_k}|^2+ \varepsilon_k^2\right)^{\frac{p-2}{2}}\,\left|\frac{\partial u_{\varepsilon_k}}{\partial x_2}\right|^2\, \eta_{R}\, dx \\
&+ \int_{\mathcal{S}_{\alpha,3R}}\,\left(|\nabla u_{\varepsilon_k}|^2+ \varepsilon_k^2\right)^{\frac{p-2}{2}}\,\left|\frac{\partial u_{\varepsilon_k}}{\partial x_2}\right|^2\, x_2\, \frac{\partial \eta_R}{\partial x_2}\, \,dx .
\end{split}
\]
Integrating by parts in the first term of the right-hand side we end up with: 
\[
\begin{split}
\int_{\mathcal{S}_{\alpha,3R}} \left\langle \nabla H_{\varepsilon_k}(\nabla u_{\varepsilon_k}), \nabla\left( \frac{\partial u_{\varepsilon_k}}{\partial x_2}\,x_2\, \eta_{R}\right) \right\rangle\, dx  &=-\int_{\mathcal{S}_{\alpha,3R}} H_{\varepsilon_k}(\nabla u_{\varepsilon_k})\, \eta_{R}\,  dx \\
&- \int_{\mathcal{S}_{\alpha,3R}} H_{\varepsilon_k}(\nabla u_{\varepsilon_k})\,x_2\,\frac{\partial\eta_R}{\partial x_2}\, dx\\
&+ \int_{\mathcal{S}_{\alpha,3R}}\left(|\nabla u_{\varepsilon_k}|^2+ \varepsilon_k^2\right)^{\frac{p-2}{2}}\,\left|\frac{\partial u_{\varepsilon_k}}{\partial x_2}\right|^2\, \eta_{R}\, dx \\
&+ \int_{\mathcal{S}_{\alpha,3R}}\,\left(|\nabla u_{\varepsilon_k}|^2+ \varepsilon_k^2\right)^{\frac{p-2}{2}}\,\left|\frac{\partial u_{\varepsilon_k}}{\partial x_2}\right|^2\,x_2\,\frac{\partial \eta_R}{\partial x_2}\,  \,dx .
\end{split}
\]
As for the right hand side in \eqref{testo}, by using that
\[
|u_{\varepsilon_k}|^{p-2}\, u_{\varepsilon_k} \,\frac{\partial u_{\varepsilon_k}}{\partial x_2 }=\frac{1}{p}\,\frac{\partial |u_{\varepsilon_k}|^p}{\partial x_2},
\] 
we have
\[
\begin{split}
\int_{\mathcal{S}_{\alpha,3R}}|u|^{p-2}\, u \,\frac{\partial u_{\varepsilon_k}}{\partial x_2 }\, x_2\, \eta_{R}\, dx
&=\frac{1}{p}\,\int_{\mathcal{S}_{\alpha,3R}}\frac{\partial |u_{\varepsilon_k}|^p}{\partial x_2}\, x_2\, \eta_{R}\, dx\\
&+ \int_{\mathcal{S}_{\alpha,3R}}(|u|^{p-2}\, u - |u_{\varepsilon_k}|^{p-2}\, u_{\varepsilon_k} )\,\frac{\partial u_{\varepsilon_k}}{\partial x_2 }\, x_2\, \eta_{R}\, dx.
\end{split}
\]
A further integration by parts leads to 
\[
\begin{split}
\int_{\mathcal{S}_{\alpha,3R}}|u|^{p-2}\, u \,\frac{\partial u_{\varepsilon_k}}{\partial x_2 }\, x_2\, \eta_{R}\, dx & =-\frac{1}{p}\int_{\mathcal{S}_{\alpha,3R}} |u_{\varepsilon_k}|^p\, \eta_{R}\, dx- \frac{1}{p}\,\int_{\mathcal{S}_{\alpha,3R}} |u_{\varepsilon_k}|^p\, x_2\,\frac{\partial\eta_R}{\partial x_2}\, dx\\
&+ \int_{\mathcal{S}_{\alpha,3R}}(|u|^{p-2}\, u - |u_{\varepsilon_k}|^{p-2}\, u_{\varepsilon_k} )\,\frac{\partial u_{\varepsilon_k}}{\partial x_2 }\, x_2\, \eta_{R}\, dx.
\end{split}
\]
By collecting together these expressions, we deduce that \eqref{testo} can be equivalently rewritten as 
\begin{equation}\label{dapassare}
\begin{split}
\frac{\lambda}{p}\,\int_{\mathcal{S}_{\alpha,3R}} |u_{\varepsilon_k}|^p\, \eta_{R}\, dx&-\int_{\mathcal{S}_{\alpha,3R}} H_{\varepsilon_k}(\nabla u_{\varepsilon_k})\, \eta_{R}\,  dx\\
&+ \int_{\mathcal{S}_{\alpha,3R}}\left(|\nabla u_{\varepsilon_k}|^2+ \varepsilon_k^2\right)^{\frac{p-2}{2}}\, \left|\frac{\partial u_{\varepsilon_k}}{\partial x_2}\right|^2\, \eta_{R}\, dx \\
&=\int_{\mathcal{S}_{\alpha,3R}} \left(H_{\varepsilon_k}(\nabla u_{\varepsilon_k})- \frac{\lambda}{p}\,|u_{\varepsilon_k}|^p\right)\, x_2\,\frac{\partial\eta_R}{\partial x_2}\, dx\\
&- \int_{\mathcal{S}_{\alpha,3R}}\,\left(|\nabla u_{\varepsilon_k}|^2+ \varepsilon_k^2\right)^{\frac{p-2}{2}}\,\left|\frac{\partial u_{\varepsilon_k}}{\partial x_2}\right|^2\,x_2\,\frac{\partial \eta_R}{\partial x_2}\,dx\\
&+\lambda \int_{\mathcal{S}_{\alpha,3R}}(|u|^{p-2} \,u - |u_{\varepsilon_k}|^{p-2}\, u_{\varepsilon_k} )\,\frac{\partial u_{\varepsilon_k}}{\partial x_2 }\, x_2\, \eta_{R}\, dx.
 \end{split}
\end{equation}
We now wish to carefully pass to the limit in \eqref{dapassare}: first as $k$ diverges to $\infty$; then as $R$ diverges to $+\infty$. 
\vskip.2cm\noindent
{\it Step 2: taking the limit as $k\to\infty$.} We start from the terms in the right-hand side. By exploiting the strong convergence \eqref{mercy} and the elementary inequalities
\[
\Big||a|^{p-2}\,a-|b|^{p-2}\,b\Big|\le \left\{\begin{array}{ll}
C_p\,|a-b|^{p-1},& \text{if}\ 1<p\le 2,\\
&\\
C_p\,\left(a^2+b^2\right)^\frac{p-2}{2}\,|a-b|,& \text{if}\ p>2,
\end{array}
\right.
\]
it is not difficult to see that
\[
\lim_{\varepsilon\to 0}  \int_{\mathcal{S}_{\alpha,3R}}(|u|^{p-2} u - |u_{\varepsilon_k}|^{p-2} u_{\varepsilon_k} )\,\frac{\partial u_{\varepsilon_k}}{\partial x_2 }\, x_2\, \eta_{R}\, dx=0.
\]
All the other terms can be easily handled by using \eqref{Hconverge}  and the strong convergence \eqref{mercy}, {\it except for} the two containing
\[
\left(|\nabla u_{\varepsilon_k}|^2+ \varepsilon_k^2\right)^{\frac{p-2}{2}}\left|\frac{\partial u_{\varepsilon_k}}{\partial x_2}\right|^2,
\]
which are the most delicate ones (at least for $p<2$). For them, we can proceed as follows: we decompose this integrand 
\[
\left(\left(|\nabla u_{\varepsilon_k}|^2+ \varepsilon_k^2\right)^{\frac{p-2}{2}}\,\frac{\partial u_{\varepsilon_k}}{\partial x_2}\right)\,\frac{\partial u_{\varepsilon_k}}{\partial x_2}=:f_k\,g_k,
\]
and observe that $\{g_k\}_{k\in\mathbb{N}}$ converges strongly in $L^p(\mathcal{S}_{\alpha,3R})$, still thanks to \eqref{mercy}. On the other hand, $\{f_k\}_{k\in\mathbb{N}}$ is uniformly bounded in $L^{p'}(\mathcal{S}_{\alpha,3R})$. Thus, it weakly conveges in $L^{p'}(\mathcal{S}_{\alpha,3R})$ to some limit function $f$, up to a subsequence. We observe that \eqref{mercy2} gives
\[
\lim_{k\to \infty} f_k=|\nabla u|^{p-2}\,\frac{\partial u}{\partial x_2}, \qquad \text{a.\,e. in}\ \mathcal{S}_{\alpha,3R}.
\]
This fact permits to identify the weak limit $f$, i.e. we have
\[
f=|\nabla u|^{p-2}\,\frac{\partial u}{\partial x_2}, \qquad \text{a.\,e. in}\ \mathcal{S}_{\alpha,3R},
\]
see for example \cite[Chapitre 1, Lemme 4.8]{Ka}. By noticing that
\[
\eta_R\in L^\infty(\mathcal{S}_{\alpha,3R})\qquad \text{and}\qquad x_2\,\frac{\partial \eta}{\partial x_2}\in L^\infty(\mathcal{S}_{\alpha,3R}),
\]
the previous discussion assures that we can safely pass to the limit in \eqref{dapassare}, as $k$ diverges to $\infty$. This gives
\begin{equation}\label{dapassare2}
\begin{split}
\frac{\lambda}{p}\,\int_{\mathcal{S}_{\alpha,3R}} |u|^p\, \eta_{R}\, dx-\frac{1}{p}\,\int_{\mathcal{S}_{\alpha,3R}} |\nabla u|^p\, \eta_{R}\,  dx&+ \int_{\mathcal{S}_{\alpha,3R}}|\nabla u|^{p-2}\,\left|\frac{\partial u}{\partial x_2}\right|^2\, \eta_{R}\, dx \\
&=\int_{\mathcal{S}_{\alpha,3R}} \left(\frac{1}{p}\,|\nabla u|^p- \frac{\lambda}{p}\,|u|^p\right)\, x_2\,\frac{\partial\eta_R}{\partial x_2}\, dx\\
&- \int_{\mathcal{S}_{\alpha,3R}}\,|\nabla u|^{p-2}\,\left|\frac{\partial u}{\partial x_2}\right|^2\,x_2\,\frac{\partial \eta_R}{\partial x_2}\,dx.
 \end{split}
\end{equation}
\vskip.2cm\noindent
{\it Step 3: taking the limit as $R\to+\infty$.} The identity \eqref{dapassare2} holds for every $R>0$.
By using the properties of $\eta_R$ and the fact that $u\in W^{1,p}_0(\mathcal{S}_{\alpha,3R})$, from the Dominated Convergence Theorem we get
\[
\begin{split}
\lim_{R\to+\infty} \left[\frac{\lambda}{p}\,\int_{\mathcal{S}_{\alpha,3R}} |u|^p\, \eta_{R}\, dx\right.&\left.-\frac{1}{p}\,\int_{\mathcal{S}_{\alpha,3R}} |\nabla u|^p\, \eta_{R}\,  dx+ \int_{\mathcal{S}_{\alpha,3R}}|\nabla u|^{p-2}\,\left|\frac{\partial u}{\partial x_2}\right|^2\, \eta_{R}\, dx \right]\\
&=\frac{\lambda}{p}\,\int_{\mathcal{S}_\alpha} |u|^p\, dx-\frac{1}{p}\,\int_{\mathcal{S}_\alpha} |\nabla u|^p\, dx+ \int_{\mathcal{S}_\alpha}|\nabla u|^{p-2}\,\left|\frac{\partial u}{\partial x_2}\right|^2\,dx. 
\end{split}
\]
As for the right-hand side of \eqref{dapassare2}, we can observe that 
\[
|\nabla u|^{p-2}\,\left|\frac{\partial u}{\partial x_2}\right|^2\in L^1(\mathcal{S}_{\alpha,R}),
\]
and
\[
\left|x_2\,\frac{\partial \eta_R}{\partial x_2}\right|\le 2\cdot\,1_{\mathcal{S}_{\alpha,2R}\setminus \mathcal{S}_{\alpha,R}},\qquad \text{a.\,e. in}\ \mathcal{S}_\alpha,
\]
thanks to the properties of $\eta_R$. Thus, again by the Dominated Convergence Theorem we get
\[
\lim_{R\to+\infty}\int_{\mathcal{S}_{\alpha,3R}}\left[ \left(\frac{1}{p}\,|\nabla u|^p- \frac{\lambda}{p}\,|u|^p\right)\, x_2\,\frac{\partial\eta_R}{\partial x_2}\, dx- \int_{\mathcal{S}_{\alpha,3R}}\,|\nabla u|^{p-2}\,\left|\frac{\partial u}{\partial x_2}\right|^2\,x_2\,\frac{\partial \eta_R}{\partial x_2}\,dx\right]=0.
\]
Finally, from \eqref{dapassare2}, we end up with
\begin{equation}
\label{dapassare3}
\frac{\lambda}{p}\,\int_{\mathcal{S}_{\alpha}} |u|^p\, dx-\frac{1}{p}\,\int_{\mathcal{S}_{\alpha}} |\nabla u|^p\, dx+ \int_{\mathcal{S}_{\alpha}}|\nabla u|^{p-2}\,\left|\frac{\partial u}{\partial x_2}\right|^2\,dx=0.
\end{equation}
{\it Step 4: conclusion.} We are only left with using the equation for $u$, which gives
\[
\int_{\mathcal{S}_{\alpha}} |\nabla u|^p\,dx=\lambda\,\int_{\mathcal{S}_{\alpha}} |u|^p\,dx.
\]
Thus, from \eqref{dapassare3} we get the claimed identity \eqref{pozzaev}. 
\end{proof}
\begin{oss}
\label{oss:qppocozallo}
By repeating verbatim the previous arguments, one could also obtain the following identity
\[
\int_{\mathcal{S}_\alpha}|\nabla u|^{p-2}\,\left|\frac{\partial u}{\partial x_2}\right|^2\, dx=\left(\frac{1}{p}-\frac{1}{q}\right)\,\int_{\mathcal{S}_\alpha} |\nabla u|^p\,dx,
\]
for every $u\in W^{1,p}_0(\Omega)\cap L^q(\Omega)$ weakly solving
\[
-\Delta u_p=\lambda\,|u|^{q-2}\,u,\qquad \text{in}\ \Omega,
\]
where
\[
\left\{\begin{array}{ll}
1<q\le p^*,& \text{if}\ 1<p<2,\\
1<q<+\infty, & \text{if}\ p\ge 2.
\end{array}
\right.
\]
We will not need this more general fact, we leave the details to the interested reader.
\end{oss}
We are ready for the main result of the section.
\begin{teo}
\label{teo:striscia}
For every $1<p<\infty$ we have
\[
\mathfrak{S}_{\mathrm{ess}, p}(\mathcal{S}_\alpha) = \left[\left(\frac{\pi_p}{2\,\alpha}\right)^p, +\infty\right)\qquad \text{and}\qquad \mathfrak{S}_{{\rm eigen},p}(\mathcal{S}_\alpha)=\emptyset.
\]
\end{teo}
\begin{proof}
We divide the proof in two parts, according to the part of the spectrum we are considering.
\vskip.2cm\noindent
{\it Part 1: eigenvalues}. Let us suppose that $\lambda\in \mathfrak{S}_{{\rm eigen},p}(\mathcal{S}_\alpha)$. By Proposition \ref{prop:pocozaev}, we get that for an associated eigenfunction  $u\in W^{1,p}(\mathcal{S}_\alpha)\setminus\{0\}$ we must have
\[
\left(\left|\frac{\partial u}{\partial x_1}\right|^2+\left|\frac{\partial u}{\partial x_2}\right|^2\right)^\frac{p-2}{2}\,\left|\frac{\partial u}{\partial x_2}\right|^2=0,\qquad \text{a.\,e. in}\ \mathcal{S}_\alpha.
\]
This in turn gives that $\partial u/\partial x_2$ must vanish almost everywhere in $\mathcal{S}_\alpha$. Thus, the function $u$ only depends on the variable $x_1$. Since $u\in L^p(\mathcal{S}_\alpha)$ and recalling the definition of $\mathcal{S}_\alpha$, this fact in particular gives that $u$ must vanish almost everywhere. This gives the desired contradiction.
\vskip.2cm\noindent
{\it Part 2: essential spectrum}. We already know that
\[
\mathfrak{S}_{\mathrm{ess},p}(\mathcal{S}_\alpha)\subseteq \big[\lambda_p(\mathcal{S}_\alpha), +\infty\big),
\] 
by the Corollary in the Introduction. Observe that
\[
\mathcal{E}_p(\mathcal{S}_\alpha) = \lambda_p(\mathcal{S}_\alpha),
\]
which follows rather easily from the definition of $\mathcal{E}_p$ and the translation invariance of $\lambda_p$.  Furthermore, we have already noted in \eqref{brunopini} that 
\[
\lambda_p(\mathcal{S}_\alpha) = \left(\frac{\pi_p}{2\,\alpha}\right)^p.
\] 
In view of the last two facts, to conclude the proof we need to prove that that for every 
\[
\lambda>\left(\frac{\pi_p}{2\,\alpha}\right)^p,
\] 
there exists a singular constrained Palais-Smale sequence at level $\lambda$. To this aim, we will suitably adapt the construction used in the proof of \cite[Proposition 6.1]{BiaBraOgn}. The main technical obstruction lies in the fact that for $p\not=2$ we can not explicitly write the first eigenpair of a rectangle: we will show below that actually we can circumvent this problem.
\par
Let $\lambda$ be such a level. We recall that the function 
\[
(0,+\infty)\ni\ell  \mapsto \lambda_p((-\alpha,\alpha)\times (-\ell,\ell)),
\]
is continuous (see for example \cite[Lemma 2.3]{Bob}) and satisfies
\[
\lim_{\ell\to 0^+}\lambda_p((-\alpha,\alpha)\times (-\ell,\ell)) = +\infty\qquad  \text{and}\qquad \lim_{\ell\to +\infty}\lambda_p((-\alpha,\alpha)\times (-\ell,\ell))=\left(\frac{\pi_p}{2\,\alpha}\right)^p.
\]
Hence, there exists $T=T(\lambda,\alpha)>0$ such that
\[
\lambda_p(\mathcal{S}_{\alpha,T})=\lambda.
\]
Since $\mathcal{S}_{\alpha,T}$ is a bounded open set, we have that $\lambda \in \mathfrak{S}_{\mathrm{eigen}, p}(\mathcal{S}_{\alpha,T})$, i.e. the bottom of the spectrum $\lambda_p(\mathcal{S}_{\alpha,T})$ is attained. Thus, there exists $\psi\in W^{1,p}_0(\mathcal{S}_{\alpha,T})\setminus\{0\}$ which satisfies
\[
\int_{\mathcal{S}_{\alpha,T}} \langle |\nabla \psi|^{p-2} \nabla \psi, \nabla \varphi\rangle\,dx= \lambda\, \int_{\mathcal{S}_{\alpha,T}}|\psi|^{p-2}\, \psi\,\varphi\,dx, \qquad \text{for every}\ \varphi\in W^{1,p}_0(\mathcal{S}_{\alpha,T}).
\] 
Of course, it is not restrictive to  assume that $\|\psi\|_{L^{p}(\mathcal{S}_{\alpha,T})}=1$. 
\par
We will construct the desired constrained Palais-Smale sequence by considering a suitable sequence of extensions of $\psi$, made of odd reflections as in Lemma \ref{lm:C1}. More precisely, 
for every $k\in\mathbb{Z}$, we consider the translated set
\[
\mathcal{S}_{\alpha,T}+2\,k\,T\,\mathbf{e}_2 = (-\alpha, \alpha) \times \big((2\,k-1)\,T, (2\,k+1)\,T\big),
\]
and define
\[
\psi_{k}(x_1,x_2) = (-1)^{|k|}\, \psi \Big(x_1, (-1)^{|k|}\, (x_2 - 2\,k\,T)\Big) , \quad \text{ for every }(x_1,x_2)\in\mathcal{S}_{\alpha,T}+2\,k\,T\,\mathbf{e}_2. 
\]
Then, for every $n\in\mathbb{N}$ we set 
\[
u_n = \left(\frac{1}{2\,n+1}\right)^{\frac{1}{p}}\sum_{k=-n}^{n} \psi_k.
\]
Observe that $u_n\in W^{1,p}_0(\mathcal{S}_{\alpha,(2\,n+1)\,T})$ and
\begin{equation}
\label{eguassione}
-\Delta_p u_n = \lambda\,|u_n|^{p-2}\,u_n,\quad  \text{in}\ \mathcal{S}_{\alpha,(2\,n+1)\,T},
\end{equation}
in weak sense.
We claim that $u_n$ is a  singular constrained Palais-Smale sequence at the level $\lambda$.
Indeed, since for every $k\in\mathbb{Z}$, it holds
\[
\psi_k\in W^{1,p}_0(\mathcal{S}_{\alpha,T}+2k\,T\,\mathbf{e}_2)\setminus\{0\},\qquad \int_{\mathcal{S}_{\alpha,T}+2k\mathbf{e}_2}|\psi_k|^p\,dx=1,\qquad \int_{\mathcal{S}_{\alpha,T}+2k\mathbf{e}_2}|\nabla \psi_k|^p\,dx = \lambda,
\]
we immediately have 
\[
u_n\in W^{1,p}_0(\mathcal{S}_{\alpha,(2\,n+1)\,T})\setminus \{0\}, \qquad \int_{\mathcal{S}_{\alpha,(2\,n+1)\,T}}|u_n|^p\,dx=1,\qquad \int_{\mathcal{S}_{\alpha,(2\,n+1)\,T}}|\nabla u_n|^p\,dx = \lambda, 
\]
for every $n\in\mathbb{N}$. Thus, the first two  assumptions in the definition of constrained Palais-Smale sequence are satisfied. 
\par
As for the third condition, we determine the equation weakly solved by $u_n$ on $\mathcal{S}_{\alpha}$. To this aim,
for every $0<\varepsilon<T$ and every $n\in\mathbb{N}$, we introduce the following piecewise linear cut-off function  $\eta_{\varepsilon,n}\in W^{1,\infty}(\mathbb{R}^2)$
\begin{equation}\label{cutoff1d}
\eta_{\varepsilon,n}(x_1,x_2)=\min\left\{\frac{(T_n-|x_2|)_+}{\varepsilon},1\right\},\qquad \text{where}\ T_n:=(2\,n+1)\,T.
\end{equation}
Thus, for every $\varphi\in W^{1,p}_0(\mathcal{S}_\alpha)$, by the Dominated Convergence Theorem we have
\begin{equation}\label{terza1}
\begin{split}
\int_{\mathcal{S}_\alpha} \langle |\nabla u_n|^{p-2}\, \nabla u_n, \nabla \varphi\rangle\,dx &=\lim_{\varepsilon\to 0}\int_{\mathcal{S}_\alpha} \langle |\nabla u_n|^{p-2}\, \nabla u_n, \nabla \varphi\rangle\, \eta_{\varepsilon, n}\, dx\\
\end{split}.
\end{equation}
On the other hand, for every $\varepsilon>0$ we can rewrite
\[
\begin{split}
\int_{\mathcal{S}_\alpha} \langle |\nabla u_n|^{p-2}\, \nabla u_n, \nabla \varphi\rangle\, \eta_{\varepsilon, n}\, dx&=\int_{\mathcal{S}_\alpha} \langle |\nabla u_n|^{p-2}\, \nabla u_n, \nabla (\varphi\, \eta_{\varepsilon, n})\rangle\, dx\\
&-\int_{\mathcal{S}_\alpha} \langle |\nabla u_n|^{p-2}\, \nabla u_n, \nabla \eta_{\varepsilon, n}\rangle\,\varphi\, dx\\
&=\lambda\,\int_{\mathcal{S}_\alpha}|u_n|^{p-2}\, u_n\, \varphi\,\eta_{\varepsilon,n}\,dx\\
&-\int_{\mathcal{S}_\alpha} \langle |\nabla u_n|^{p-2}\, \nabla u_n, \nabla \eta_{\varepsilon, n}\rangle\,\varphi\, dx.
\end{split}
\]
In the second equality we used the weak formulation of  \eqref{eguassione}, tested with $\varphi\,\eta_{\varepsilon,n}\in W^{1,p}_0(\mathcal{S}_{\alpha,T_n})$. We observe that, again by the Dominated Convergence Theorem, we have 
\[
\lim_{\varepsilon\to 0}\int_{\mathcal{S}_\alpha}|u_n|^{p-2}\, u_n\, \varphi\,\eta_{\varepsilon,n}\,dx=\int_{\mathcal{S}_\alpha}|u_n|^{p-2}\, u_n\, \varphi\,dx.
\]
As for the integral containing $\nabla \eta_{\varepsilon,n}$: we use the definition of $u_n$, that of the cut off-function \eqref{cutoff1d} and Fubini's theorem, so to obtain
\[
\begin{split}
\int_{\mathcal{S}_\alpha} \langle |\nabla u_n|^{p-2}\, \nabla u_n, \nabla \eta_{\varepsilon, n}\rangle\,\varphi\, dx &=
\left(\frac{1}{2n+1}\right)^{\frac{p-1}{p}}\,\frac{1}{\varepsilon}\int_{-T_n}^{-T_n+\varepsilon} \left(\int_{-\alpha}^\alpha |\nabla u_n|^{p-2}\,\frac{\partial u_n}{\partial x_2}\, \varphi\,dx_1\right)\,dx_2\\
&-\left(\frac{1}{2n+1}\right)^{\frac{p-1}{p}}\,\frac{1}{\varepsilon}\int_{T_n-\varepsilon}^{T_n} \left(\int_{-\alpha}^\alpha |\nabla u_n|^{p-2}\,\frac{\partial u_n}{\partial x_2}\, \varphi\,dx_1\right)\,dx_{2}.
\end{split}
\]
Observe that $u_n\in C^1(\overline{\mathcal{S}_{\alpha,T_n}})$, by Lemma \ref{lm:C1}. Thus, the previous quantity admits limit as $\varepsilon$ goes to $0$, given by
\[
\begin{split}
\kappa_n & := \left(\frac{1}{2n+1}\right)^{\frac{p-1}{p}}\, \int_{-\alpha}^\alpha |\nabla u_n(x_1, T_n)|^{p-2}\,\frac{\partial u_n}{\partial x_2}(x_1, T_n)\, \varphi(x_1, T_n) \,dx_1\\
&-\left(\frac{1}{2n+1}\right)^{\frac{p-1}{p}}\,\int_{-\alpha}^\alpha |\nabla u_{n}(x_1, -T_n)|^{p-2}\,\frac{\partial u_{n}(x_1, -T_n)}{\partial x_2}\, \varphi(x_1, T_{-n})\,dx_1.
\end{split}
\]
By \eqref{terza1} and the above discussion, we can thus obtain that 
\begin{equation}\label{terza2}
\begin{split}
\left\langle -\Delta_p u_n -\lambda\, |u_n|^{p-2}\,u_n, \varphi\right\rangle_{(W^{-1,p'}(\mathcal{S}_\alpha), W^{1,p}(\mathcal{S}_\alpha))}=\kappa_n.
\end{split}
\end{equation}
By using the trace inequality of Lemma \ref{lm:traccia} below, from the definition of $\kappa_n$ we get
\[
\begin{split}
|\kappa_n| & \le 4\,\alpha^\frac{p-1}{p}\, \left(\frac{1}{2n+1}\right)^{\frac{p-1}{p}}\,\|\nabla\psi\|_{L^\infty(\mathcal{S}_{\alpha,T})}^{p-1}\, \|\varphi\|_{W^{1,p}(\mathcal{S}_\alpha)}.
\end{split}
\]
By plugging this inequality into \eqref{terza2}, we finally get
\[
\left\|-\Delta_p u_n -\lambda\, |u_n|^{p-2}\,u_n\right\|_{W^{-1,p'}(\mathcal{S}_\alpha)}\le 4\,\alpha^\frac{p-1}{p}\,\left(\frac{1}{2\,n+1}\right)^{\frac{p-1}{p}} \,\|\nabla\psi\|_{L^\infty(\mathcal{S}_{\alpha,T})}^{p-1}.
\]
This shows that $\{u_n\}_{n\in\mathbb{N}}$ is a constrained Palais-Smale sequence at the level $\lambda$. In order to conclude, we still have to prove that the sequence is singular, that is 
\[
\lim_{n\to\infty}\|u_n\|_{L^{p}(\mathcal{S}_\alpha\cap B_R)}=0,\qquad \text{for every}\ R>0.
\]
This is simple: for every $R>0$, we have that there exists $n_R\in\mathbb{N}$ such that
\[
\mathcal{S}_\alpha\cap B_R\subseteq \mathcal{S}_{\alpha,T_{n_R}}.
\]
Accordingly, by construction we get
\[
\|u_n\|_{L^{p}(\mathcal{S}_\alpha\cap B_R)}\le \left(\frac{1}{2\,n+1}\right)^{\frac{1}{p}}\,(2\,n_R+1).
\]
By taking the limit as $n$ diverges to $\infty$, we get the desired conclusion. This concludes the proof.
\end{proof}
\begin{oss}
We notice that the previous argument to exclude existence of eigenfunctions in $\mathcal{S}_\alpha$ can be generalized to prove that 
\[
-\Delta_p u=\lambda\,|u|^{q-2}\,u,\qquad \text{in}\ \mathcal{S}_\alpha,
\]
does not admit any nontrivial solution $u\in W^{1,p}_0(\mathcal{S}_\alpha)\cap L^q(\mathcal{S}_\alpha)$, in the {\it sub-homogeneous regime} $1<q<p$.
Indeed, in light of Remark \ref{oss:qppocozallo}, we get 
\[
\int_{\mathcal{S}_\alpha}|\nabla u|^{p-2}\,\left|\frac{\partial u}{\partial x_2}\right|^2\, dx\le 0,
\]
thanks to the fact that $1<q<p$. This is still sufficient to replicate the argument above. 
\par 
We recall that {\it this non-existence result in the strip is sharp}. Indeed, for a sub-critical exponent $q>p$ we do have existence of non-trivial solutions (see \cite[Theorem 7.5 \& Corollary 7.6]{AmTo} and \cite[Theorem 5.1]{BBT} for the case $p=2$, \cite[Theorem 5.1]{BraBriPri_steiner} for the general case).
\end{oss}
In the previous result, we used the following trace inequality. We include the proof, for completeness. 
\begin{lm}
\label{lm:traccia}
Let $1\le p<\infty$, for every $T\in\mathbb{R}$ we have
\[
\|\varphi\|_{L^1((-\alpha,\alpha)\times\{T\})}\le 2\,\alpha^\frac{p-1}{p}\,\|\varphi\|_{W^{1,p}(\mathcal{S}_\alpha)},\qquad \text{for every}\ \varphi\in C^\infty_0(\mathcal{S}_\alpha).
\]
\end{lm}
\begin{proof}
For every $T\in\mathbb{R}$ and every $(x_1,x_2)\in(-\alpha,\alpha)\times(T,T+1)$ we have
\[
\begin{split}
|\varphi(x_1,T)|&\le |\varphi(x_1,x_2)-\varphi(x_1,T)|+|\varphi(x_1,x_2)|\\
&\le \int_T^{x_2} \left|\frac{\partial \varphi}{\partial x_2}(x_1,t)\right|\,dt+|\varphi(x_1,x_2)|\\
&\le \int_T^{T+1} \left|\frac{\partial \varphi}{\partial x_2}(x_1,t)\right|\,dt+|\varphi(x_1,x_2)|.
\end{split}
\]
We now integrate with respect to both $x_1\in(-\alpha,\alpha)$ and $x_2\in(T,T+1)$. This yields
\[
\begin{split}
\int_{-\alpha}^\alpha|\varphi(x_1,T)|\,dx_1&\le \iint_{(-\alpha,\alpha)\times (T,T+1)} \left|\frac{\partial \varphi}{\partial x_2}(x_1,t)\right|\,dx_1\,dt\\
&+\iint_{(-\alpha,\alpha)\times (T,T+1)}|\varphi(x_1,x_2)|\,dx_1\,dx_2.
\end{split}
\]
We can now use H\"older's inequality in the right-hand side, so to obtain
\[
\begin{split}
\int_{-\alpha}^\alpha|\varphi(x_1,T)|\,dx_1&\le  (2\,\alpha)^\frac{p-1}{p}\,\left(\iint_{(-\alpha,\alpha)\times (T,T+1)} \left|\frac{\partial \varphi}{\partial x_2}(x_1,t)\right|^p\,dx_1\,dt\right)^\frac{1}{p}\\
&+(2\,\alpha)^\frac{p-1}{p}\,\left(\iint_{(-\alpha,\alpha)\times (T,T+1)}|\varphi(x_1,x_2)|^p\,dx_1\,dx_2\right)^\frac{1}{p}.
\end{split}
\]
Thus, we obtain the desired estimate, by using that $a^{1/p}+b^{1/p}\le 2^{1/p}\,(a+b)^{1/p}$.
\end{proof}

\appendix

\section{A strong vanishing property}
\label{sec:A}
In the proof of Lemma \ref{lm:2}, we crucially exploited the following result. The proof is elementary, though some care is needed.
\begin{prop}
\label{prop:salvaculo}
Let $\{u_n\}_{n\ge 1}$ be the sequence constructed in the proof of Lemma \ref{lm:2}. Then we have
\[
\lim_{n\to\infty} \int_\Omega V_n^{k-1}\,|\nabla u_n|^p\,dx=0=\lim_{n\to\infty} \int_\Omega V_n^k\,|u_n|^p\,dx,\qquad \text{for every}\ k\in\mathbb{N}\setminus\{0\}.
\]
\end{prop}
\begin{proof}
The proof will use an induction argument and it is a slight variation of the argument used in \cite[Proposition 5.4]{BraBriPri_periodic}, containing a simpler potential term $V_n$. Despite the close similarity, we prefer to give all the details, in order to assist the reader. This will also clarify the role of the parameters used to define the potential $V_n$ and how to handle the additional difficulties coming from the presence of the term $\zeta_n$ in the definition of $V_n$.
\par
We take $R\ge 1$ and  $k\in\mathbb N$. We then use
\[
\varphi = V_n^k\,u_n\,\zeta_R^p,
\]
as a test function in \eqref{EL}. We are still using the notation $\zeta_R$ for the cut-off function as in {\it Step 1} of the proof of Lemma \ref{lm:2}, i.e.
 $\zeta_R\in C_0^\infty(B_{2R})$, with 
\[
0\le\zeta\le1,\qquad \zeta_R=1\ \text{in}\ B_R,\qquad |\nabla \zeta_R|\le \frac{2}{R}.
\]
Observe that the previous test function is feasible, thanks to the compact support of $\zeta_R$, which ``hides'' the growth at infinity of the term $V_n^k$.
We then obtain
\begin{equation}
\label{1}
\begin{split}
\int_\Omega |\nabla u_n|^p\,V_n^k\,\zeta_R^p\,dx+\int_\Omega V_n^{k+1}\,|u_n|^p\,\zeta_R^p\,dx&=-p\,\int_\Omega \langle |\nabla u_n|^{p-2}\,\nabla u_n,\nabla \zeta_R\rangle\,\zeta_R^{p-1}\,u_n\,V_n^k\,dx\\
&-k\,\int_\Omega \langle |\nabla u_n|^{p-2}\,\nabla u_n,\nabla V_n\rangle\,V_n^{k-1}\,u_n\,\zeta_R^p\,dx\\
&+\lambda_n\,\int_\Omega |u_n|^p\,V_n^k\,\zeta_R^p\,dx.
\end{split}
\end{equation}
We estimate the first term on the right-hand side by means of Cauchy-Schwarz and Young inequalities. This yields
\[
\begin{split}
-p\,\int_\Omega \langle |\nabla u_n|^{p-2}\,\nabla u_n,\nabla \zeta_R\rangle\,\zeta_R^{p-1}\,u_n\,V_n^k\,dx&\le \delta\,(p-1)\,\int_\Omega V_n^k\, |\nabla u_n|^p\,\zeta_R^{p}\,dx\\
&+\delta^{1-p}\,\int_\Omega V_n^k\,|u_n|^p\,|\nabla\zeta_R|^p\,dx,
\end{split}
\]
for every $\delta>0$. Thus, from \eqref{1} we get
\begin{equation}
\label{2}
\begin{split}
(1-(p-1)\,\delta)\,\int_\Omega |\nabla u_n|^p\,V_n^k\,\zeta_R^p\,dx&+\int_\Omega V_n^{k+1}\,|u_n|^p\,\zeta_R^p\,dx\\
&\le -k\,\int_\Omega \langle |\nabla u_n|^{p-2}\,\nabla u_n,\nabla V_n\rangle\,V_n^{k-1}\,u_n\,\zeta_R^p\,dx\\
&+ \lambda_n\,\int_\Omega V_n^k\,|u_n|^p\,\zeta_R^p\,dx\\
&+\delta^{1-p}\,\int_\Omega V_n^k\,|u_n|^p\,|\nabla\zeta_R|^p\,dx.
\end{split}
\end{equation}
We now estimate the term containing $\nabla V_n$. By taking into account that\footnote{Here we see the role of the factor $\sqrt{n}$ in front of the term $\zeta_n$, in the definition of $V_n$. The term $\sqrt{n}\,\zeta_n$ is becoming larger and larger as $n$ grows to $\infty$, so to ``confine'' the optimizer $u_n$ ``towards infinity''. At the same time, the gradient of the term $\sqrt{n}\,\zeta_n$ is harmless.}
\[
|\nabla V_n|=\left|\sqrt{n}\,\nabla \zeta_n+\frac{1}{R_n^2}\,\frac{x}{|x|}\right|\le \sqrt{n}\,|\nabla\zeta_n|+\frac{1}{R_n^2}\le \frac{2}{\sqrt{n}}+\frac{1}{R_n^2},
\] 
we have
\[
\begin{split}
-k\,\int_\Omega \langle |\nabla u_n|^{p-2}\,\nabla u_n,\nabla V_n\rangle\,V_n^{k-1}\,u_n\,\zeta_R^p\,dx&\le k\,\left(\frac{2}{\sqrt{n}}+\frac{1}{R_n^2}\right)\,\int_\Omega |\nabla u_n|^{p-1}\,|u_n|\,V_n^{k-1}\,\zeta_R^p\,dx\\
&\le k\,\left(\frac{2}{\sqrt{n}}+\frac{1}{R_n^2}\right)\,\frac{p-1}{p}\,\int_\Omega |\nabla u_n|^p\,V_n^{k-1}\,\zeta_R^p\,dx\\
&+k\,\left(\frac{2}{\sqrt{n}}+\frac{1}{R_n^2}\right)\,\frac{1}{p}\,\int_\Omega |u_n|^p\,V_n^{k-1}\,\zeta_R^p\,dx.
\end{split}
\]
We use this estimate in \eqref{2}, so to obtain
\[
\begin{split}
(1-(p-1)\,\delta)\,\int_\Omega |\nabla u_n|^p\,V_n^k\,\zeta_R^p\,dx&+\int_\Omega V_n^{k+1}\,|u_n|^p\,\zeta_R^p\,dx\\
&\le k\,\left(\frac{2}{\sqrt{n}}+\frac{1}{R_n^2}\right)\,\frac{p-1}{p}\,\int_\Omega |\nabla u_n|^p\,V_n^{k-1}\,\zeta_R^p\,dx\\
&+k\,\left(\frac{2}{\sqrt{n}}+\frac{1}{R_n^2}\right)\,\frac{1}{p}\,\int_\Omega |u_n|^p\,V_n^{k-1}\,\zeta_R^p\,dx\\
&+ \lambda_n\,\int_\Omega V_n^k\,|u_n|^p\,\zeta_R^p\,dx\\
&+\delta^{1-p}\,\int_\Omega V_n^k\,|u_n|^p\,|\nabla\zeta_R|^p\,dx.\end{split}
\]
This is valid for every $\delta>0$: we can take $\delta=1/(2\,(p-1))$ and use the properties of $\zeta_R$, so to obtain
\[
\begin{split}
\frac{1}{2}\,\int_{\Omega\cap B_R} |\nabla u_n|^p\,V_n^k\,dx+\int_{\Omega\cap B_R} V_n^{k+1}\,|u_n|^p\,dx&\le k\,\left(\frac{2}{\sqrt{n}}+\frac{1}{R_n^2}\right)\,\frac{p-1}{p}\,\int_{\Omega\cap B_{2R}} |\nabla u_n|^p\,V_n^{k-1}\,dx\\
&+k\,\left(\frac{2}{\sqrt{n}}+\frac{1}{R_n^2}\right)\,\frac{1}{p}\,\int_{\Omega\cap B_{2R}} |u_n|^p\,V_n^{k-1}\,dx\\
&+\lambda_n\,\int_{\Omega\cap B_{2R}} V_n^k\,|u_n|^p\,dx\\
&+\frac{2^p}{R^p}\,(2\,(p-1))^{p-1}\,\int_{\Omega\cap B_{2R}} V_n^k\,|u_n|^p\,dx.
\end{split}
\]
On the right-hand side, we can also use the uniform bound \eqref{uniforme} on $\lambda_n$ and the fact that
\[
V_n^{k-1}\le 1+V_n^{k}.
\]
With simple manipulations, we can thus obtain (recall that $R\ge 1$)
\[
\begin{split}
\int_{\Omega\cap B_R} |\nabla u_n|^p\,V_n^k\,dx+\int_{\Omega\cap B_R} V_n^{k+1}\,|u_n|^p\,dx&\le C\,k\,\left(\frac{1}{\sqrt{n}}+\frac{1}{R_n^2}\right)\,\int_{\Omega\cap B_{2R}} |\nabla u_n|^p\,V_n^{k-1}\,dx\\
&+C\,\left[k\,\left(\frac{1}{\sqrt{n}}+\frac{1}{R_n^2}\right)+1\right]\,\int_{\Omega\cap B_{2R}} |u_n|^p\,V_n^{k}\,dx\\
&+C\,k\,\left(\frac{1}{\sqrt{n}}+\frac{1}{R_n^2}\right)\,\int_{\Omega\cap B_{2R}} |u_n|^p\,dx\\
\end{split}
\]
for a constant $C=C(N,p,\Omega)>0$. Finally, we can simply estimate the last $L^p$ norm by recalling that each $u_n$ has unit $L^p$ norm.
We now introduce the following compact notation
\[
\mathcal{I}_{n,k}(R)=\int_{\Omega\cap B_R} V_n^k\,|u_n|^p\,dx\quad \text{and}\quad \mathcal{J}_{n,k}(R)=\int_{\Omega\cap B_R} |\nabla u_n|^p\,V_n^{k-1}\,dx,\qquad n\in\mathbb{N},\ k\in\mathbb{N}\setminus\{0\}.
\]
Thus, we can reformulate the previous estimate as follows
\begin{equation}
\label{daghelo}
\begin{split}
\mathcal{I}_{n,k+1}(R)+\mathcal{J}_{n,k+1}(R)&\le C\,k\,\left(\frac{1}{\sqrt{n}}+\frac{1}{R_n^2}\right)\,\mathcal{J}_{n,k}(2R)\\
&+C\,\left[k\,\left(\frac{1}{\sqrt{n}}+\frac{1}{R_n^2}\right)+1\right]\,\mathcal{I}_{n,k}(2R)+C\,k\,\left(\frac{1}{\sqrt{n}}+\frac{1}{R_n^2}\right).
\end{split}
\end{equation}
We can now proceed by induction on $k\in\mathbb{N}\setminus\{0\}$ and obtain that
\[
\mathcal{I}_{n,k}:=\int_\Omega V_n^{k}\,|u_n|^p\,dx<+\infty\qquad \text{and}\qquad \mathcal{J}_{n,k}:=\int_\Omega V_n^{k-1}\, |\nabla u_n|^p\,dx<+\infty.
\]
Indeed, this is true for $k=1$ by construction. On the other hand, if we suppose that both $\mathcal{I}_{n,k}$ and $\mathcal{J}_{n,k}$ are finite for a certain $k\ge 1$, by taking the limit as $R$ goes to $+\infty$ in \eqref{daghelo} and relying on the Monotone Convergence Theorem, we get  
\begin{equation}
\label{sprot}
\begin{split}
\mathcal{I}_{n,k+1}+\mathcal{J}_{n,k+1}&\le C\,k\,\left(\frac{1}{\sqrt{n}}+\frac{1}{R_n^2}\right)\,\mathcal{J}_{n,k}\\
&+C\,\left[k\,\left(\frac{1}{\sqrt{n}}+\frac{1}{R_n^2}\right)+1\right]\,\mathcal{I}_{n,k}+C\,k\,\left(\frac{1}{\sqrt{n}}+\frac{1}{R_n^2}\right).
\end{split}
\end{equation}
The desired conclusion can now be obtained from \eqref{sprot}, by using again an induction argument over $k$. Indeed, observe that by \eqref{panino} we have
\[
\lim_{n\to\infty}\mathcal{I}_{n,1}=\lim_{n\to\infty} \int_\Omega V_n\,|u_n|^p\,dx=0,
\]
while obviously
\[
\begin{split}
\lim_{n\to\infty} \left[C\,k\,\left(\frac{1}{\sqrt{n}}+\frac{1}{R_n^2}\right)\,\mathcal{J}_{n,1}\right.&\left.+C\,k\,\left(\frac{2}{\sqrt{n}}+\frac{1}{R_n^2}\right)\right]\\
&=\lim_{n\to\infty} C\,k\,\left(\frac{1}{\sqrt{n}}+\frac{1}{R_n^2}\right)\,\int_\Omega |\nabla u_n|^p\,dx=0,
\end{split}
\]
thanks to \eqref{uniforme} and the fact that $\{R_n\}_{n\in\mathbb{N}}$ diverges to $+\infty$. This concludes the proof.
\end{proof}


\medskip

\begin{thebibliography}{100}

\bibitem{An} A. Anane, Simplicit\'e et isolation de la premi\`ere valeur propre du $p-$laplacien avec poids, C. R. Acad. Sci. Paris S\'er. I Math., {\bf 305} (1987), 725--728.

\bibitem{AmTo} C. J. Amick, J. F. Toland, Nonlinear elliptic eigenvalue problems on an infinite strip -- global theory of bifurcation and asymptotic bifurcation, Math. Ann., {\bf 262} (1983), 313--342.

\bibitem{BiaBraOgn} F. Bianchi, L. Brasco, R. Ognibene, On the spectrum of sets made of cores and tubes, J. Convex Anal., {\bf 31} 2024, 315--358.

\bibitem{BS} M. Sh. Birman, M. Z. Solomjak, {\it Spectral theory of selfadjoint operators in Hilbert space.}
Translated from the 1980 Russian original by S. Khrushchev and V. Peller. Mathematics and its Applications (Soviet Series). D. Reidel Publishing Co., Dordrecht, 1987.

\bibitem{Bob} V. Bobkov, Non-Ljusternik--Schnirelman eigenvalues of the pure $p-$Laplacian exist, preprint (2026), {\tt https://arxiv.org/abs/2604.01138}

\bibitem{BBT} J. L. Bona, D. K. Bose, R. E. L. Turner, Finite-amplitude steady waves in stratified fluids,
J. Math. Pures Appl. (9), {\bf 62} (1983), 389--439.

\bibitem{BraBook} L. Brasco,  {\it Handbook of Calculus of Variations for absolute beginners.} Unitext, {\bf 163}, La Mat. per il 3+2, Springer, Cham, 2025.

\bibitem{Bra_bus}  L. Brasco, On principal frequencies and isoperimetric ratios in convex sets, Ann. Fac. Sci. Toulouse Math. (6), {\bf 29} (2020), 977--1005. 

\bibitem{BraBriPri_low} L. Brasco, L. Briani, F. Prinari, Low eigenvalues of the $p-$Laplacian in general open sets, preprint (2026), {\tt https://arxiv.org/abs/2602.21118}

\bibitem{BraBriPri_periodic} L. Brasco, L. Briani, F. Prinari, Extremals for sharp Poincar\'e-Sobolev inequalities: periodically perforated sets and beyond, preprint (2025), {\tt https://arxiv.org/abs/2511.20260}

\bibitem{BraBriPri_steiner} L. Brasco, L. Briani, F. Prinari, Extremals for Poincar\'e-Sobolev sharp constants in Steiner symmetric sets, Nonlinear Anal., {\bf 269} (2026), Paper No. 114098.

\bibitem{CDeFG} M. Cuesta, D. Figueiredo, J.-P. Gossez, The beginning of the Fu\v{c}ik spectrum for the $p-$Laplacian, J. Differential Equations, {\bf 159} (1999), 212--238.

\bibitem{DiB} E. DiBenedetto, $C^{1+\alpha}$ local regularity of weak solutions of degenerate elliptic equations, Nonlinear Anal., {\bf 7} (1983), 827--850.


\bibitem{EL} M. J. Esteban, P.-L. Lions, Existence and nonexistence results for semilinear elliptic problems in unbounded domains,
Proc. Roy. Soc. Edinburgh Sect. A, {\bf 93} (1982/83), 1--14.

\bibitem{Fra} G. Franzina, {\it Existence, Uniqueness, Optimization and Stability for low Eigenvalues of some Nonlinear Operators}, PhD Thesis (2012), available at {\tt https://cvgmt.sns.it/paper/2102/}

\bibitem{Fri} L. Friedlander, Asymptotic behavior of the eigenvalues of the $p-$Laplacian, Comm. Partial Differential Equations, {\bf 14} (1989), 1059--1069.

\bibitem{GP} J. P. Garc\'ia Azorero, I. Peral Alonso, Existence and nonuniqueness for the $p-$Laplacian: nonlinear eigenvalues, Comm. Partial Differential Equations, {\bf 12} (1987), 1389--1430.

\bibitem{GV} M. Guedda, L. V\'eron, Quasilinear elliptic equations involving critical Sobolev exponents, Nonlinear Anal., {\bf 13} (1989), 879--902.

\bibitem{HS} P. D. Hislop, I. M. Sigal, {\it Introduction to spectral theory. With applications to Schr\"odinger operators}, Appl. Math. Sci., {\bf 113}. Springer-Verlag, New York, 1996

\bibitem{Ka}
O.\ Kavian, Introduction \`a la th\'eorie des points critiques et applications aux probl\`emes elliptiques. (French) [Introduction to critical point theory and applications to elliptic problems] Mathématiques \& Applications (Berlin) [Mathematics \& Applications], {\bf 13}. Springer-Verlag, Paris, 1993.

\bibitem{KT} H. Koch, D. Tataru, Carleman estimates and absence of embedded eigenvalues, Comm. Math. Phys., {\bf 267} (2006), 419--449.

\bibitem{Li} E. H. Lieb, On the lowest eigenvalue of the Laplacian for the intersection of two domains,
Invent. Math., {\bf 74} (1983), 441--448.

\bibitem{Le} A. L\^{e},  Eigenvalue problems for the $p-$Laplacian, Nonlinear Anal., {\bf 64} (2006), 1057--1099.

\bibitem{Lin} P. Lindqvist, A nonlinear eigenvalue problem, in {\it Topics in mathematical analysis}, 175--203. Ser. Anal. Appl. Comput., {\bf 3}. World Scientific Publishing Co. Pte. Ltd., Hackensack, NJ, 2008.

\bibitem{Maz} V. Maz'ya, {\it Sobolev spaces with applications to elliptic partial differential equations}. Second, revised and augmented edition. Grundlehren der Mathematischen Wissenschaften [Fundamental Principles of Mathematical Sciences], {\bf 342}. Springer, Heidelberg, 2011. 

\bibitem{Pe} A. Persson, Bounds for the discrete part of the spectrum of a semi-bounded Schr\"odinger operator, Math. Scand., {\bf 8} (1960), 143--153.

\bibitem{Re} F. Rellich, \"Uber das asymptotische Verhalten der L\"osungen von $\Delta u+\lambda u=0$ in unendlichen Gebieten, Jber. Deutsch. Math.-Verein., {\bf 53} (1943), 57--65.

\bibitem{Ro} S. N. Roze, The spectrum of a second order elliptic operator, Mat. Sb., {\bf 80 (122)} (1969), 195--209.


\bibitem{St} M. Struwe, {\it Variational methods. Applications to nonlinear partial differential equations and Hamiltonian systems. Fourth edition}. Ergebnisse der Mathematik und ihrer Grenzgebiete. 3. Folge. A Series of Modern Surveys in Mathematics, {\bf 34}. Springer-Verlag, Berlin, 2008.

\bibitem{Szu} A. Szulkin, Ljusternik-Schnirelmann theory on $C^1$-manifolds, Ann. Inst. H. Poincar\'e Anal. Non Lin\'eaire, {\bf 5} (1988), 119--139.

\bibitem{Te} G. Teschl, {\it Mathematical methods in quantum mechanics. With applications to Schr\"odinger operators.} Second edition. Graduate Studies in Mathematics, {\bf 157}. American Mathematical Society, Providence, RI, 2014.

\bibitem{Wi} K. J. Witsch, Examples of embedded eigenvalues for the Dirichlet-Laplacian in domains with infinite boundaries, Math. Methods Appl. Sci., {\bf 12} (1990), 177--182.

\end{thebibliography}
\end{document}